\documentclass[11pt,reqno]{amsart}
\usepackage{amsmath,amssymb,amsthm,mathtools,mathrsfs}
\usepackage[T1]{fontenc}
\usepackage{lmodern,microtype}
\usepackage{needspace}
\microtypesetup{expansion=false}
\usepackage[colorlinks=true,linkcolor=black,citecolor=black,urlcolor=black]{hyperref}
\hypersetup{pdftitle={The Sharp Isoperimetric Inequality in the First Heisenberg Group},pdfauthor={Wei-Chun Chen, Wei-Ting Kao, and Chun-Kai Lien}}
\makeatletter
\def\@settitle{\begin{center}%
  \normalfont\Large\bfseries\baselineskip20pt\relax
  \@title
  \end{center}}
\def\@setauthors{%
  \begingroup
  \def\thanks{\protect\thanks@warning}%
  \trivlist
  \centering\small\scshape\@topsep24pt\relax
  \advance\@topsep by -\baselineskip
  \item\relax
  \author@andify\authors
  \def\\{\protect\linebreak}%
  \authors
  \endtrivlist
  \endgroup}
\def\section{\@startsection{section}{1}%
  \z@{1.1\linespacing\@plus.3\linespacing}{.6\linespacing}%
  {\normalfont\large\bfseries\raggedright}}
\def\subsection{\@startsection{subsection}{2}%
  \z@{.5\linespacing\@plus.7\linespacing}{-.5em}%
  {\normalfont\bfseries}}
\makeatother
\newtheoremstyle{paperplain}
  {9pt plus 2pt minus 1pt}{6pt plus 1pt minus 1pt}
  {\itshape}{}{\bfseries}{.}{.5em}
  {\thmname{#1}\thmnumber{ #2}\thmnote{\normalfont\ (#3)}}
\theoremstyle{paperplain}
\newtheorem{theorem}{Theorem}[section]
\newtheorem{proposition}[theorem]{Proposition}
\newtheorem{lemma}[theorem]{Lemma}
\theoremstyle{definition}\newtheorem*{AIuse}{AI Usage Statement}
\theoremstyle{remark}\newtheorem{remark}[theorem]{Remark}
\numberwithin{equation}{section}
\allowdisplaybreaks[2]
\AtBeginDocument{%
  \setlength{\abovedisplayskip}{9pt plus 2pt minus 2pt}%
  \setlength{\belowdisplayskip}{9pt plus 2pt minus 2pt}%
  \setlength{\abovedisplayshortskip}{4pt plus 2pt}%
  \setlength{\belowdisplayshortskip}{7pt plus 2pt minus 2pt}%
  \clubpenalty=10000
  \widowpenalty=10000
  \displaywidowpenalty=10000}
\AddToHook{env/theorem/before}{\Needspace{6\baselineskip}}
\AddToHook{env/proposition/before}{\Needspace{6\baselineskip}}
\AddToHook{env/lemma/before}{\Needspace{6\baselineskip}}
\newcommand{\HH}{\mathbb H^1}
\newcommand{\RR}{\mathbb R}
\newcommand{\Div}{\operatorname{div}}
\newcommand{\supp}{\operatorname{supp}}
\newcommand{\eps}{\varepsilon}
\newcommand{\nuphase}{\widehat{\nu}}
\newcommand{\angI}{\mathcal I}
\newcommand{\EE}{\mathcal E}
\newcommand{\NN}{\mathcal N}
\newcommand{\ip}[2]{\langle #1,#2\rangle}
\title[Isoperimetry in the first Heisenberg group]{The Sharp Isoperimetric Inequality\\in the First Heisenberg Group}
\author{Wei-Chun Chen}
\address{Department of Mathematics, National Taiwan University, Taipei 10617, Taiwan}
\email{b14201028@ntu.edu.tw}

\author{Wei-Ting Kao}
\address{Institute of Mathematics, Academia Sinica, Taipei 106319, Taiwan}
\email{wtkao@as.edu.tw}

\author{Chun-Kai Lien}
\address{Department of Mathematics, National Taiwan University, Taipei 10617, Taiwan}
\email{d13221002@ntu.edu.tw}
\date{}
\begin{document}
\begin{abstract}
We prove the sharp horizontal isoperimetric inequality in the first Heisenberg group for all sets of finite horizontal perimeter and positive finite volume, with equality precisely for left translates and dilates of the Pansu ball.
\end{abstract}
\maketitle

\section{Introduction}\label{sec:main}
The horizontal isoperimetric problem in the first Heisenberg group asks for the least horizontal perimeter among sets of prescribed Haar measure. Pansu proposed the rotationally symmetric spheres as the boundaries of its minimizers \cite[p.~172]{Pansu}. Isoperimetric regions exist and admit bounded open representatives \cite{LR}. Ritor\'e and Rosales classified isoperimetric regions with $C^2$ boundary as Pansu balls \cite[Theorem~7.2]{RitoreRosales}. We establish the unrestricted sharp inequality and its rigidity statement without assuming boundary regularity of an arbitrary horizontal minimizer.

For a horizontal minimizer whose boundary may be nonsmooth, the main difficulty is to obtain smooth approximations retaining the second-variation inequalities needed for rigidity. We obtain them by subtracting the isoperimetric profile from an elliptic perimeter and adding a vanishing penalty.

We identify $\HH$ with $\RR^3$ and use the group law
\[
 (x,y,t)\cdot(x',y',t')=(x+x',y+y',t+t'+yx'-xy').
\]
Its left-invariant horizontal orthonormal frame, contact form, and dilations are
\begin{gather*}
 X=\partial_x+y\partial_t,\qquad Y=\partial_y-x\partial_t,
 \qquad T=\partial_t,\qquad [X,Y]=-2T,\\
 \Theta=dt+x\,dy-y\,dx,\qquad
 \delta_s(x,y,t)=(sx,sy,s^2t).
\end{gather*}
Haar measure is $d\mathcal L^3=dx\,dy\,dt$. For a measurable set $E$, its horizontal perimeter is
\[
 P_H(E)=|D_H\mathbf1_E|(\HH),\qquad D_H\mathbf1_E=(D_X\mathbf1_E,D_Y\mathbf1_E).
\]
Here $\langle D_Z\mathbf1_E,\varphi\rangle=-\int_E Z\varphi\,d\mathcal L^3$ for $Z=X,Y$ and $\varphi\in C_c^\infty(\HH)$; $|D_H\mathbf1_E|$ denotes total variation, and $P_H(E)=+\infty$ if $D_H\mathbf1_E$ is not a finite vector-valued Radon measure. We identify sets that differ by a Haar-null set and use the finite-perimeter framework of \cite{AGM}.

In these coordinates the unit Pansu ball is
\begin{equation}\label{eq:unitmodel}
 \Omega=\{(z,t):|z|<1,\ |t|<g(|z|)\},\qquad
 g(r)=\tfrac12\bigl(\arccos r+r\sqrt{1-r^2}\bigr).
\end{equation}
\begin{theorem}[Sharp isoperimetry]\label{thm:main}
Every measurable set $E\subset\HH$ with $0<|E|<\infty$ and $P_H(E)<\infty$ satisfies
\begin{equation}\label{eq:sharp}
 P_H(E)^4\ge\frac{512\pi^2}{27}|E|^3.
\end{equation}
Equality holds if and only if there exist $\mathbf a\in\HH$ and $s>0$ such that
\[
 |E\triangle(\mathbf a\cdot\delta_s\Omega)|=0.
\]
\end{theorem}

The homogeneous dimension four determines the powers in the inequality: $P_H(\delta_sE)=s^3P_H(E)$ and $|\delta_sE|=s^4|E|$. Direct graph integration for $\Omega$, using $g'(r)=-r^2/\sqrt{1-r^2}$, gives
\begin{equation}\label{eq:unitball}
 P_H(\Omega)=4\pi\int_0^1\frac{r^2}{\sqrt{1-r^2}}\,dr=\pi^2,
 \qquad |\Omega|=4\pi\int_0^1rg(r)\,dr=\frac{3\pi^2}{8}.
\end{equation}
Consequently $P_H(\Omega)^4/|\Omega|^3=512\pi^2/27$.

Normalize an arbitrary minimizer $E$ so that $3P=8V$, where $P=P_H(E)$ and $V=|E|$. The regularized perimeter $P_\eps$ supplies ellipticity, and the isoperimetric profile gives the least horizontal perimeter at each volume. Subtracting the profile from $P_\eps$ and adding a vanishing pinning penalty gives smooth stable approximations converging to $E$. The reconstructed centre at each noncharacteristic boundary point specifies a left translate of the unit Pansu sphere passing through that point. The estimate in Proposition~\ref{prop:prescribedestimate} makes the centre derivatives along characteristic arcs tend to zero and controls the transverse terms in the second variation.

Nonnegativity of the full Hessian of the penalized functional, tested with speeds depending on the centre and with contact variations, forces the reconstructed centres to concentrate at a single point. Horizontal divergence tests then identify $E$ as a Pansu ball. Section~\ref{sec:selection} constructs the approximations, and Section~\ref{sec:rigidity} proves this rigidity statement.

The centre reconstruction follows \cite{CentreRigidity}, whose quantitative estimate is proved for $n\ge2$ under volume-constrained stability. The smooth bump construction used for the penalized approximations in \cite[Sections~3.1--3.2]{HigherIsoperimetry} is applied here on both sides of the boundary to control the symmetric difference. Lemma~\ref{lem:residual} proves the prescribed-curvature remainder bound in $\HH$ uniformly in $\delta/\eps$. Proposition~\ref{prop:prescribedestimate} combines this bound with the nonnegative full Hessian and controls the approximations as their volume coefficient approaches the critical profile value, without a condition on the relative decay rates of the curvature error and $\eps$.

\Needspace{7\baselineskip}
\begin{AIuse}
The proof strategy in this paper is based primarily on the authors' earlier
work in \cite{CentreRigidity,HigherIsoperimetry}. ChatGPT 6.0 Astra assisted
in refining the authors' original proof strategy and in carrying out the
calculations. The authors independently reviewed and verified all
mathematical arguments and calculations in the manuscript, and take full
responsibility for its correctness and content.
\end{AIuse}

\section{Centre reconstruction and the second-variation estimate}\label{sec:prelim}
For $V>0$, nonsharp horizontal isoperimetry and existence \cite{LR}, together with dilation invariance, give
\begin{equation}\label{eq:profile}
 I(V):=\inf_{|F|=V}P_H(F)=I(1)V^{3/4}>0.
\end{equation}
For a minimizer of volume $V$ and perimeter $P$, the scaling laws $P\mapsto s^3P$ and $V\mapsto s^4V$ show that dilation by $s=3P/(8V)$ normalizes the new perimeter and volume by $3P=8V$.
The normalized minimizers will be identified as left translates of $\Omega$.

\begin{remark}[Coordinate normalization]
The normalization in \cite{LR} uses the vertical coordinate $t'=2t$: the map $\iota(z,t)=(z,2t)$ intertwines the group laws and sends $(X,Y)$ to the orthonormal frame
$X'=\partial_x+2y\partial_{t'}$,
$Y'=\partial_y-2x\partial_{t'}$.
Its Euclidean Jacobian is $2$, so $|\iota(E)|=2|E|$ and $P'(\iota(E))=2P_H(E)$, where $P'$ is the horizontal perimeter for $(X',Y')$. Thus the existence and representative results apply with the present normalization.
\end{remark}

For $\eps>0$, define the left-invariant Riemannian metrics
\[
 g_\eps=dx^2+dy^2+\eps^{-2}\Theta^2.
\]
Thus $(X,Y,\eps T)$ is $g_\eps$-orthonormal. Unless specified otherwise, Riemannian geometric quantities below are computed with $g_1$. Set $JX=Y$, $JY=-X$, $JT=0$. On a smooth oriented boundary $\Sigma=\partial F$ with outward $g_1$-unit normal $N$, let $A(W)=\nabla_WN$ for $W\in T\Sigma$ and write
\[
 N=\rho\nu+{N_T} T,\qquad \rho=\cos u,\qquad{N_T}=\sin u,
 \qquad -\pi/2\le u\le\pi/2.
\]
On the noncharacteristic set $\rho>0$, the tangent frame
$Z=J\nu$, $U=\rho T-{N_T}\nu$ is positively oriented. Write $d\sigma_{g_\eps}$ for the surface area measure induced by $g_\eps$. The horizontal perimeter measure is $d\sigma_H=\rho\,d\sigma_{g_1}$.
On the regular part of the unit Pansu sphere, $H=\Div_H\nu=2$ and $Zu=-1$. For a general surface, write
\begin{equation}\label{eq:fields}
 H=\Div_H\nu,\quad e_u=Zu+1,\quad \kappa=Uu,\quad e_H=H-2.
\end{equation}
Put $\angI=(-\pi/2,\pi/2)$.
Identify horizontal vectors $aX+bY$ with $a+ib$. Since $|\nu|=1$, write $\nu=\cos\theta+i\sin\theta=e^{i\theta}$ locally, where $\theta$ is the horizontal normal angle, and put $\nuphase=e^{2iu}\nu=e^{i(\theta+2u)}$.
With inward horizontal normal $-\nu$ and characteristic direction $Z$, the centre reconstruction in \cite[Section~2.3]{StableRigidity} and \cite{CentreRigidity} is the pair $(z_0,t_0)$ given by
\begin{equation}\label{eq:centre}
 z_0=z-\tfrac12(\nu+\nuphase),\qquad
 t_0=t+\ip{Jz_0}{z}-\tfrac12(u+{N_T}\rho).
\end{equation}
For $\mathbf a(p)=(z_0(p),t_0(p))$, equation~\eqref{eq:centre} gives
\begin{equation}\label{eq:incidence}
 \mathbf a(p)^{-1}\cdot p
 =\left(\rho e^{iu}\nu,\frac{u+\rho\sin u}{2}\right)
 \in\partial\Omega,
 \qquad p\in\partial(\mathbf a(p)\cdot\Omega).
\end{equation}
Indeed, the horizontal coordinate has modulus $\rho$, and
$(u+\rho\sin u)/2=\operatorname{sign}(u)g(\rho)$.
As $u\to\pm\pi/2$, the right-hand side tends to $(0,\pm\pi/4)$ with error $O(|u\mp\pi/2|)$, uniformly for $|\nu|=1$.

Centres reconstructed from surfaces in a common compact set remain in a common compact set.

Write $z_0=x_0+iy_0$. The horizontal component of the centre derivative in a tangent direction $\xi$ is $\xi z_0$, while its vertical component in the left-invariant frame at $\mathbf a$ is
\[
 \Theta_{\mathbf a}(\xi\mathbf a)
 =\xi t_0+x_0\xi y_0-y_0\xi x_0.
\]
The extra terms account for the vertical component introduced by left translation.

\begin{subequations}\label{eq:centreD}
Along a characteristic curve, $Z\theta=H$ is the turning rate of the horizontal normal, and $Zu=-1+e_u$ is the rate of change of its inclination. The errors $e_H=H-2$ and $e_u$ measure the departures from the corresponding rates on the unit sphere. Differentiating the centre gives
\begin{equation}\label{eq:centreDZ}
 \begin{aligned}
 Zz_0&=-J\bigl[e_H(z-z_0)+e_u\nuphase\bigr],\\
 \Theta_{\mathbf a}(Z\mathbf a)&=\rho^2e_H.
 \end{aligned}
\end{equation}
If $e_H=e_u=0$ along such a curve, both components vanish and its reconstructed centre is constant.

The perpendicular tangent direction $U$ measures variation across characteristic curves. Here $Uu=\kappa$, and the centre derivative is
\begin{equation}\label{eq:centreDU}
 \begin{aligned}
 Uz_0&=-J\bigl[\tfrac{e_u}{\rho}(z-z_0)+(\kappa-\rho)\nuphase\bigr],\\
 \Theta_{\mathbf a}(U\mathbf a)&=\rho e_u.
 \end{aligned}
\end{equation}
On the unit sphere $\kappa=\rho$, so $\kappa-\rho$ measures the departure of the transverse inclination rate from the model. These identities distinguish constancy of the centre along each characteristic curve from agreement of centres across different curves.

The horizontal formulas follow from \eqref{eq:centre} and
$(\nu+\nuphase)/2=\rho^2\nu+{N_T}\rho J\nu$.
For the vertical components, symmetry of the shape form gives
$U\theta=(e_u-1+2{N_T}^2)/\rho$.
\end{subequations}

Put $d=z-z_0=\rho e^{iu}\nu$. For either tangent field $\xi\in\{Z,U\}$, the identities $\tfrac12\xi(u+{N_T}\rho)=\rho^2\xi u$ and $\ip{Jd}{\xi d}=\rho^2(\xi\theta+\xi u)$ give
\begin{align*}
 \xi t_0&=\xi t+\ip{J\xi z_0}{z}+\ip{Jz_0}{\xi z}-\rho^2\xi u,\\
 \Theta_{\mathbf a}(\xi\mathbf a)
 &=\Theta(\xi)-2\ip{Jd}{\xi z}+\rho^2\xi\theta.
\end{align*}
Using $Zz=J\nu$, $Uz=-{N_T}\nu$, $\Theta(Z)=0$ and $\Theta(U)=\rho$, the two values are
\[
 -2\rho^2+\rho^2H=\rho^2e_H,\qquad
 \rho-2{N_T}^2\rho+\rho^2U\theta=\rho e_u.
\]

For a smooth bounded region $F$ and $\eps>0$, define
\begin{equation}\label{eq:ellipticperimeter}
 \begin{gathered}
 P_\eps(F)=\eps\operatorname{Area}_{g_\eps}(\Sigma)
          =\int_\Sigma\Phi_\eps\,d\sigma_{g_1},\\
 \Phi_\eps=(\rho^2+\eps^2{N_T}^2)^{1/2},\qquad m=\rho/\Phi_\eps.
 \end{gathered}
\end{equation}
The volume and surface area measures satisfy
\[
 dV_{g_\eps}=\eps^{-1}d\mathcal L^3,\qquad
 d\sigma_{g_\eps}=\eps^{-1}\Phi_\eps\,d\sigma_{g_1}.
\]
For a set $F$ of finite $g_1$-perimeter, write $\partial^*F$ for its reduced boundary, with its measure-theoretic outer $g_1$-unit normal. Define $P_\eps$ by the same convex normal integrand on $\partial^*F$.

The first variation of $P_\eps$ defines $H_\eps$: a $g_1$-normal speed $w$ has variation $\int_{\Sigma} H_\eps w\,d\sigma_{g_1}$.

The energy density and its integral are
\begin{gather}
 e=|Zz_0|^2+e_u^2
   =\rho^2(e_H+e_u)^2+(2-\rho^2)e_u^2,\notag\\
 \EE=\int_{\Sigma} e\,d\sigma_H,\qquad
 \NN_\eps=\eps^2\int_{\Sigma}\frac{\kappa^2}{\Phi_\eps^2}\,d\sigma_H.\label{eq:e}
\end{gather}
Thus $\EE$ controls $Zz_0$ and $Zu+1$; the differential formulas are given in \eqref{eq:centreD}. Completing the square gives
\begin{equation}\label{eq:coercivity}
 e\ge\frac{\rho^2}{4}e_H^2+\frac12e_u^2.
\end{equation}
Thus $\EE$ controls the $L^2(d\sigma_{g_1})$ norms of $e_H$ and $e_u$ on each set $\{\rho\ge\vartheta\}$ with fixed $\vartheta>0$.

Put $w_F=\rho+u{N_T}$. On the unit Pansu sphere centred at the identity, this is the normal speed of Heisenberg dilation \cite[Section~3.2]{StableRigidity}. Indeed, \eqref{eq:incidence} gives $z=\rho e^{iu}\nu$ and $2t=u+\rho N_T$, so the dilation field $D=x\partial_x+y\partial_y+2t\partial_t=xX+yY+2tT$ satisfies
\[
 \langle D,N\rangle_{g_1}
 =\rho^3+(u+\rho N_T)N_T
 =\rho+uN_T=w_F.
\]
On a general surface, the same function of $u$ has an integral that bounds total $g_1$-area, which controls the contact second variations in Section~\ref{sec:rigidity}.

\begin{lemma}[Area bounds and an integral identity]\label{lem:areacontrol}
On every smooth closed oriented surface $\Sigma$, put $I_F=\int_\Sigma w_F\,d\sigma_{g_1}$. Then
\begin{equation}\label{eq:areaflux}
\begin{gathered}
 0=\int_{\Sigma}\Div_\Sigma(u\rho Z)\,d\sigma_{g_1}
   =\int_{\Sigma}(\rho(e_u-1)+2u{N_T})\,d\sigma_{g_1},\\
 2\int_{\Sigma} w_F\,d\sigma_{g_1}=3P_H-\int_{\Sigma} e_u\,d\sigma_H.
\end{gathered}
\end{equation}
\begin{equation}\label{eq:areabound}
 \operatorname{Area}_{g_1}(\Sigma)\le I_F\le\tfrac12\bigl(3P_H+\sqrt{P_H\EE}\bigr).
\end{equation}
\end{lemma}
\begin{proof}
The tangent field $W=JN=\rho Z$ is smooth globally and satisfies $|W|=\rho$. Near a characteristic point, the sign $s$ of ${N_T}$ is constant and
\[
 u\rho Z=uW=s\tfrac\pi2W-s(\arcsin\rho)W.
\]
Since $\rho=|W|$ is Lipschitz, the last term is $O_\Sigma(\rho^2)$ with derivative $O_\Sigma(\rho)$ away from the zero set. It therefore extends in $C^1$ with zero derivative at every characteristic point. Thus $u\rho Z$ is a global $C^1$ tangent field. On the noncharacteristic set,
\[
 \Div_\Sigma(\rho Z)=2{N_T},\qquad
 \Div_\Sigma(u\rho Z)=\rho(e_u-1)+2u{N_T}.
\]
The divergence theorem gives \eqref{eq:areaflux}. Since $dw_F/du=u\rho$, one has $w_F\ge w_F(0)=1$. Cauchy--Schwarz and $e_u^2\le e$ give \eqref{eq:areabound}.
\end{proof}

\begin{proposition}[Critical stability estimate]\label{prop:prescribedestimate}
Let $\mathcal U$ be open and let $F\Subset\mathcal U$ be a smooth region with $P_H(F)>0$. For $0<\eps\le1$, suppose $F$ is stationary for
\[
 P_\eps(F)-\int_F(2+r)\,d\mathcal L^3.
\]
Let $\mathcal B_{2+r}$ denote the second variation of this functional. Suppose that, for some $c\ge0$,
\begin{equation}\label{eq:estimateform}
 \mathcal B_{2+r}(v,v)
 +c\left(\int_{\partial F}v\,d\sigma_{g_1}\right)^2\ge0
\end{equation}
for every smooth normal speed $v$. Put
\[
 \delta=\|r\|_{C^1(\mathcal U)},\qquad
 q=\eps^2+\delta+\left(cP_H(F)-\frac43\right)_+,
\]
where $s_+=\max\{s,0\}$ and derivatives are measured using $g_1$. There are absolute constants $C,q_0>0$ such that
\begin{equation}\label{eq:robustCE}
 \EE+\NN_\eps\le CqP_H(F)\qquad\hbox{if }q\le q_0.
\end{equation}
\end{proposition}
The critical coefficient is the one supplied by the isoperimetric profile: for a normalized minimizer with $3P=8V$ and $P=I(V)$, equation~\eqref{eq:profile} gives
\[
 -I''(V)=\frac{3P}{16V^2}=\frac4{3P}.
\]
Subtracting $I(|F|)$ therefore provides the limiting volume term needed in \eqref{eq:estimateform}. For $q\le1$ and $P=P_H(F)$, the test speed $w_F$ combines the area identity \eqref{eq:areaflux} with the critical coefficient $4/(3P)$. The cancellation identity \eqref{eq:estimatecancellation} and the residual bound \eqref{eq:Rbound} give
\[
 \begin{aligned}
 0&\le\mathcal B_{2+r}(w_F,w_F)+cI_F^2\\
  &\le-\EE-\NN_\eps+C_*q(P+\EE),
 \end{aligned}
\]
where $C_*$ is absolute. The proof in Section~\ref{sec:quantitativeestimate} derives this inequality and absorbs the $q\EE$ term to obtain \eqref{eq:robustCE}. Lemma~\ref{lem:killinghessian} supplies the second-variation formula. Section~\ref{sec:selection} constructs approximations for which $q\to0$.

\subsection{Second variation and admissible normal speeds}\label{sec:secondvariation}

Near a characteristic point, $\Sigma$ is a vertical graph $t=g(x,y)$ with characteristic set $G^{-1}(0)$, where $G=(g_x-y,g_y+x)$. The rank estimate in \cite[proof of Theorem~D, (6.9)--(6.10)]{CHY} gives $\operatorname{rank}DG\ge1$. The implicit function theorem therefore places the characteristic set locally in a smooth curve, so $\sigma_{g_1}\{\rho=0\}=0$.

\begin{lemma}[Second variation formula]\label{lem:killinghessian}
Let ${F}\Subset\HH$ be a smooth region in $(\HH,g_1)$ and put $\Sigma=\partial F$. Write $T^\top=T-{N_T}N$, $H_1=\operatorname{tr}A$, and $\mathcal K=|A|^2+\operatorname{Ric}(N,N)$. For fixed $\eps>0$ and an ambient potential $h\in C^1$, stationarity for $P_\eps(F)-\int_Fh\,d\mathcal L^3$ is equivalent to
\begin{equation}\label{eq:killingfirst}
 h=H_\eps=\Phi_\eps H_1+
 \Div_\Sigma\left(\frac{(\eps^2-1){N_T}}{\Phi_\eps}T^\top\right).
\end{equation}
At such a surface, the Hessian on normal speeds is
$\mathcal B_h(w_1,w_2)=\mathcal B_\eps(w_1,w_2)-\int_\Sigma(Nh)w_1w_2\,d\sigma_{g_1}$, where
\begin{equation}\label{eq:killinghessian}
 \begin{split}
 \mathcal B_\eps(w_1,w_2)=\int_\Sigma\biggl[&
 \frac{\ip{\nabla w_1}{\nabla w_2}}{\Phi_\eps}
 +\frac{\eps^2-1}{\Phi_\eps^3}(T^\top w_1)(T^\top w_2)\\
 &-\left(\frac{\mathcal K}{\Phi_\eps}
 +\frac{\eps^2-1}{\Phi_\eps^3}|\nabla{N_T}|^2\right)w_1w_2\biggr]d\sigma_{g_1}.
 \end{split}
\end{equation}
\end{lemma}

The geometric form $\mathcal B_\eps$ is defined on every smooth surface. Its coefficients are globally smooth, including at characteristic points. The frame reduction below derives the local formula directly.

\begin{proof}
For a smooth normal speed $w$, take $\mathcal X_\tau(p)=\exp_p(\tau w(p)N(p))$ and let $J_\tau$ be its $g_1$-area Jacobian. Put $f(s)=\sqrt{1+(\eps^2-1)s^2}$. Differentiation at zero gives
\begin{gather*}
 \dot N=-\nabla w,\qquad \dot J=H_1w,\qquad
 \ddot J=|\nabla w|^2+(H_1^2-\mathcal K)w^2,\\
 \dot{N_T}=-T^\top w,\qquad
 \ddot{N_T}=-{N_T}|\nabla w|^2+2w\ip{\nabla{N_T}}{\nabla w}.
\end{gather*}
For completeness, all derivatives of $N$ along the variation are covariant. Normal orthogonality gives
\[
 \ddot N=-|\nabla w|^2N+2wA(\nabla w),
 \qquad \nabla_\Sigma{N_T}=A(T^\top)-\nabla_NT.
\]
The second identity uses the Killing equation; in particular $\nabla_NT$ is tangent to $\Sigma$. Along each normal geodesic, the Killing equation also gives $\langle N,\nabla_N\nabla_NT\rangle=0$, with $N$ parallel along that geodesic. Consequently
\begin{align*}
 \ddot{N_T}
 &=\langle\ddot N,T\rangle
       +2w\langle\dot N,\nabla_NT\rangle\\
 &=-{N_T}|\nabla w|^2
       +2w\langle A(T^\top)-\nabla_NT,\nabla w\rangle,
\end{align*}
which is the displayed normal-component formula. These identities hold globally because $N$ and $T$ are smooth. The first variation has density $fH_1+\Div_\Sigma(f'T^\top)$, proving \eqref{eq:killingfirst}.

In the second variation subtract $\int_\Sigma(Nh+H_1h)w^2d\sigma_{g_1}$. One integration by parts gives gradient coefficients $f-{N_T}f'$ and $f''$, and potential
\[
 f\mathcal K+Nh+\Div_\Sigma(f'\nabla{N_T})-f'T^\top H_1.
\]
The Killing flow of $T$ and the ordinary mean-curvature variation give
\[
 \Delta_\Sigma{N_T}+\mathcal K{N_T}=T^\top H_1,
\]
so the potential becomes $(f-{N_T}f')\mathcal K+f''|\nabla{N_T}|^2+Nh$. Since
\[
 f-{N_T}f'=\Phi_\eps^{-1},\qquad
 f''=(\eps^2-1)\Phi_\eps^{-3},
\]
polarization proves the formula.
\end{proof}

The Koszul formula and $[X,Y]=-2T$ give the nonzero connection coefficients
\begin{gather*}
 \nabla_XY=-T,\qquad \nabla_YX=T,\\
 \nabla_XT=\nabla_TX=Y,\qquad
 \nabla_YT=\nabla_TY=-X;
\end{gather*}
the remaining coefficients in this frame vanish. We use
$\mathcal R(L,M)=\nabla_L\nabla_M-\nabla_M\nabla_L-\nabla_{[L,M]}$.
Thus the sectional curvatures of the $XY$, $XT$ and $YT$ planes are $-3$, $1$ and $1$, respectively. With $N=\rho\nu+{N_T}T$, this gives
\[
 \operatorname{Ric}(N,N)=2-4\rho^2.
\]
On the noncharacteristic set, direct frame differentiation gives
\[
 A=\begin{pmatrix}\rho H&e_u\\e_u&\kappa\end{pmatrix},
 \quad \nabla_\Sigma u=AU-Z,
 \quad H_\eps=mH+b_\eps \kappa,
 \quad b_\eps=\frac{\eps^2}{\Phi_\eps^3}.
\]
Put $q_N=\rho^2(H^2-4)+(e_u+1)^2$. Equation \eqref{eq:killinghessian} in this frame becomes
\begin{equation}\label{eq:hessian}
 \mathcal B_\eps(w_1,w_2)=\int_{\Sigma}\left[
 \Phi_\eps^{-1}Zw_1Zw_2+b_\eps Uw_1Uw_2
 -(q_N/\Phi_\eps+b_\eps(\kappa^2+(e_u-1)^2))w_1w_2\right]d\sigma_{g_1}.
\end{equation}

The global formula \eqref{eq:killinghessian} has bounded coefficients on the compact surface at each fixed $\eps>0$. Thus $\mathcal B_h$ is continuous on $H^1(\Sigma)\times H^1(\Sigma)$. For a global $C^1$ speed $w$, take smooth $u_j\to w$ in $H^1$. The surface integral is continuous on $H^1(\Sigma)$, so \eqref{eq:estimateform} passes to the limit and holds for $w$.

\begin{lemma}[Extension of weighted tangent fields]\label{lem:globalflux}
On each smooth compact surface without boundary, the fields ${N_T}\rho^2(e_u-1)Z$ and
$e_H{N_T}\rho^3U$ extend by zero to $C^1$ tangent fields. Their derivatives vanish at every characteristic point. The field ${N_T}\rho^3U$ is smooth globally.
\end{lemma}
\begin{proof}
Set $W=JN=\rho Z$ and recall $T^\top=T-{N_T} N=\rho U$. These are smooth global tangent fields with $|W|=|T^\top|=\rho$. The shape operator $A$ is smooth. On $\rho>0$,
\[
 \ip{AW}{T^\top}=\rho^2e_u,\qquad \ip{AW}{W}=\rho^3H.
\]
Consequently the two fields are
\[
 F_1=\frac{N_T}\rho\bigl(\ip{AW}{T^\top}-\rho^2\bigr)W,
 \qquad
 F_2={N_T}\left(\frac{\ip{AW}{W}}\rho-2\rho^2\right)T^\top.
\]
The divided numerators $\ip{AW}{T^\top}-\rho^2$ and $\ip{AW}{W}$ are $O_\Sigma(\rho^2)$, with derivatives $O_\Sigma(\rho)$; indeed $W,T^\top=O_\Sigma(\rho)$ and $A$ is smooth. Since $\rho=|W|$ is Lipschitz, differentiation of the displayed expressions on the complement of its zero set gives
\[
 |F_i|\le C_\Sigma\rho^2,\qquad
 |\nabla_\Sigma F_i|\le C_\Sigma\rho.
\]
At a characteristic point $p$, $\rho(x)\le C_\Sigma\operatorname{dist}_\Sigma(x,p)$. The zero extension therefore has derivative zero at $p$, and the derivative estimate proves continuity. Finally ${N_T}\rho^3U={N_T}\rho^2T^\top$ is smooth because $\rho^2=1-{N_T}^2$ is smooth.

On $\rho>0$, the matrix of $A$ in the orthonormal frame $(Z,U)$ gives the fixed-surface bounds
\[
 |e_u-1|+|\kappa|+|\rho H|\le C_\Sigma.
\]
These bounds make $e$ and $q_N$ bounded on the whole surface. The global $C^1$ tangent fields admit integration by the ordinary divergence theorem. The constants $C_\Sigma$ here depend on the fixed surface and are used only for extension and integrability.

\end{proof}

\subsection{Second variation on \texorpdfstring{$w_F,N_T,w_K$}{w\_F, N\_T, w\_K}}\label{sec:weightedvariation}
Vertical translation has normal speed $N_T$. To combine its mixed second variation with the self-pairing of $w_F$, choose $w_K$ so that $w_F^2-N_Tw_K=\rho^4$ and the derivative cancellation in \eqref{eq:contract} holds. This leads to
\[
 w_K=2u\rho+(u^2+\rho^2){N_T}.
\]
Their derivatives with respect to the global scalar ${N_T}$ are:
\[
 \frac{dw_F}{d{N_T}}=u,\qquad
 \frac{dw_K}{d{N_T}}=u^2+3\rho^2.
\]
Both derivatives extend continuously to ${N_T}=\pm1$, so $w_F$ and $w_K$ are globally $C^1$ and can be inserted into the second variation by its $H^1$ continuity.

At a surface with $H_\eps=2+r$, the mixed second variation with the vertical translation speed is
\[
 \mathcal B_{2+r}(N_T,w)=-\int_\Sigma w\,Tr\,d\sigma_{g_1}.
\]
Thus this mixed term vanishes when $r=0$. The area identity \eqref{eq:areaflux} gives $2I_F=P_H(3-\bar e_u)$, where $\bar e_u=P_H^{-1}\int_\Sigma e_u\,d\sigma_H$. Consequently the critical volume term $4I_F^2/(3P_H)$ cancels the constant and linear terms in \eqref{eq:integrated}, leaving the squared errors in \eqref{eq:estimatecancellation}.

For a symmetric bilinear form put
\[
 [\mathcal B]_*=\mathcal B(w_F,w_F)-\mathcal B({N_T},w_K).
\]
This combination need not be nonnegative.
The identity $w_F^2-{N_T}w_K=\rho^4$ and $(\partial_uw_F)^2-(\partial_u{N_T})(\partial_uw_K)=-3\rho^4$, with $dw/du=\rho\,dw/d{N_T}$, give
\begin{equation}\label{eq:contract}
 [\mathcal B_\eps]_*
  =-\int_{\Sigma}\frac{\rho^4}{\Phi_\eps}(q_N+3(e_u-1)^2)\,d\sigma_{g_1}
    -\int_{\Sigma}\frac{\eps^2\rho^4}{\Phi_\eps^3}(4\kappa^2+(e_u-1)^2)\,d\sigma_{g_1}.
\end{equation}
To combine \eqref{eq:contract} with the critical volume term, separate $-\EE$ from the $e_u$ and perimeter terms in \eqref{eq:integrated}; the remaining integral retains the negative transverse term $-\NN_\eps$ by Lemma~\ref{lem:residual}. The regularization correction is $j_\eps$, and the remainder after removing the divergence is $R_\eps$. Set
\begin{align*}
 j_\eps&=\rho^2(1-m)(\rho^2H^2+4e_u^2-4e_u+4{N_T}^2)
              -\eps^2m^3((e_u-1)^2+4\kappa^2),\\
 R_\eps&=j_\eps+e_H\bigl[\rho^2{N_T}^2H+\rho(3-4\rho^2)\kappa
                            -2\rho^2(2-e_u)\bigr].
\end{align*}
Here $\rho H=\ip{AZ}{Z}$ and $\kappa=\ip{AU}{U}=Uu$ are the normal curvatures in the $Z$- and $U$-directions.
The fixed-surface bounds in Lemma~\ref{lem:globalflux} make these densities integrable for every fixed $\eps>0$. Rearranging \eqref{eq:contract} gives
\begin{align*}
 [\mathcal B_\eps]_*
 &=-\int_{\Sigma}\rho^3(q_N+3(e_u-1)^2)\,d\sigma_{g_1}+\int_{\Sigma} j_\eps\,d\sigma_H,\\
 \rho^2(q_N+3(e_u-1)^2)
 &=\rho^2(\rho^2H^2+4e_u^2-4e_u+4{N_T}^2).
\end{align*}
The local identities are
\begin{align*}
 Ze_u&=\rho UH-\frac{N_T}\rho
       [2e_u(e_u+1)+\rho^2(H^2-4)],\\
 \Div_\Sigma({N_T}\rho^3U)
    &=-\rho[\rho^2{N_T}^2H+\rho(3-4\rho^2)\kappa].
\end{align*}
For the first identity, $\nabla^\Sigma_ZZ={N_T}H\,U$ and
$\nabla^\Sigma_UZ={N_T}(e_u+1)U/\rho$ give
\[
 [Z,U]=-{N_T}H\,Z-\frac{N_T}{\rho}(e_u+1)U.
\]
Apply $Z(U\theta)-U(Z\theta)=[Z,U]\theta$ to the local horizontal angle.
Using $Z\theta=H$, $U\theta=(e_u-1+2{N_T}^2)/\rho$ from \eqref{eq:centreD},
$Z\rho=-{N_T}(e_u-1)$ and $Z{N_T}=\rho(e_u-1)$ gives the displayed formula for $Ze_u$.
All divided frame formulas are used only on $\rho>0$.
Together with \eqref{eq:fields} and \eqref{eq:e}, these identities give
\begin{align*}
 -\rho^3(q_N+3(e_u-1)^2)
    &=-\rho(e+e_u^2-2e_u+3)+\rho(R_\eps-j_\eps)\\
    &\quad+\Div_\Sigma(e_H{N_T}\rho^3U-{N_T}\rho^2(e_u-1)Z).
\end{align*}
Integrating, using Lemma~\ref{lem:globalflux} to eliminate the divergence term, and adding $\int_{\Sigma} j_\eps\,d\sigma_H$ gives
\begin{equation}\label{eq:integrated}
 \begin{split}
 [\mathcal B_\eps]_*
 &=-\EE-\int_{\Sigma} e_u^2\,d\sigma_H+2\int_{\Sigma} e_u\,d\sigma_H-3P_H
       +\int_{\Sigma} R_\eps\,d\sigma_H.
 \end{split}
\end{equation}

\subsection{Residual estimate}
\begin{lemma}[Uniform residual estimate]\label{lem:residual}
Let $0<\eps\le1$ and let a smooth compact oriented surface $\Sigma$ without boundary satisfy $H_\eps=2+r$, with $|r|\le\delta$ and $|\nabla_\Sigma r|\le\delta\le1$. Then
\begin{equation}\label{eq:Rbound}
 \int_\Sigma R_\eps\,d\sigma_H
 \le-\NN_\eps+C(\eps^2+\delta)(P_H+\EE),
\end{equation}
where $C$ is an absolute constant.
\end{lemma}
\begin{proof}
The estimate is uniform in $\delta/\eps$. The negative quadratic term in $\kappa$ absorbs the regularization terms linear in $\kappa$. Directly absorbing $r\rho(3-4\rho^2)\kappa$ would introduce a $\delta^2/\eps^2$ loss; instead, integration by parts transfers its derivative to $r$, using the bound on $\nabla_\Sigma r$.

The energy decomposition and $H=(e_H+e_u)+(2-e_u)$ give
\[
 e_u^2+\rho^2H^2\le C(1+e).
\]
Together with $0\le\rho^2(1-m)\le\eps^2{N_T}^2/2$, this yields
\begin{equation}\label{eq:jexplicit}
 j_\eps\le-4\eps^2m^3\kappa^2+C\eps^2(1+e).
\end{equation}

Put $h=2+r$ and $d_\eps=h(\Phi_\eps-\rho)$. The prescribed curvature identity $h=mH+\eps^2\Phi_\eps^{-3}\kappa$ gives
\[
 \rho(H-h)=d_\eps-\frac{\eps^2}{\Phi_\eps^2}\kappa,
 \qquad |d_\eps|\le\frac{3\eps^2}{\Phi_\eps}.
\]
Writing $\xi=\rho(H-h)$ in the definition of $R_\eps$ gives
\[
 \begin{split}
 R_\eps={}&j_\eps+{N_T}^2\xi^2
       +2\rho({N_T}^2-2+e_u)\xi+(3-4\rho^2)\xi\kappa\\
 &+r\rho(3-4\rho^2)\kappa
       +r\rho^2[{N_T}^2(2H-2-r)-2(2-e_u)].
 \end{split}
\]
The last term has absolute value at most $C\delta(1+e)$. Substituting the formula for $\xi$ and using \eqref{eq:jexplicit}, the coefficient of $\kappa^2$ in the resulting upper bound is
\[
 -\frac{\eps^2}{\Phi_\eps^2}
       (2-4\rho^2+m^2+4\rho^2m).
\]
Since $m\ge\rho$,
\[
 2-4\rho^2+m^2+4\rho^2m
 \ge2-3\rho^2+4\rho^3
 =\tfrac74+(\rho-\tfrac12)^2(4\rho+1)\ge\tfrac74.
\]
Apart from $r\rho(3-4\rho^2)\kappa$, the linear coefficient is
\[
 L=(3-4\rho^2)d_\eps
       -\frac{2\eps^2}{\Phi_\eps^2}
          [{N_T}^2d_\eps+\rho({N_T}^2-2+e_u)].
\]
The bounds $\rho\le\Phi_\eps$, $\eps\le\Phi_\eps$ and $|d_\eps|\le3\eps^2/\Phi_\eps$ give
$|L|\le C\eps^2\Phi_\eps^{-1}(1+|e_u|)$; the constant term has absolute value at most $C(\eps^2+\delta)(1+e)$. Squaring the bound gives
\[
 L^2\le C\eps^4\Phi_\eps^{-2}(1+e_u^2).
\]
Dividing by the coefficient $\eps^2\Phi_\eps^{-2}$ in Young's inequality leaves $C\eps^2(1+e)$, uniformly in $\rho$ and $\eps$. Thus
\[
 |L\kappa|\le\frac{3\eps^2}{4\Phi_\eps^2}\kappa^2+C\eps^2(1+e),
\]
and hence
\begin{equation}\label{eq:residualdensity}
 R_\eps\le-\frac{\eps^2\kappa^2}{\Phi_\eps^2}
       +C(\eps^2+\delta)(1+e)+r\rho(3-4\rho^2)\kappa.
\end{equation}

For the remaining $r\kappa$ term, use
$-({N_T}\rho^3)'=\rho^2(3-4\rho^2)$ and $\Div_\Sigma U=-{N_T} H$,
\begin{equation}\label{eq:rA}
 \begin{split}
 \int_{\Sigma} r\rho(3-4\rho^2)\kappa\,d\sigma_H
 &=-\int_{\Sigma} r\,U({N_T}\rho^3)\,d\sigma_{g_1}\\
 &=\int_{\Sigma}[{N_T}\rho^2Ur-r{N_T}^2\rho^2H]\,d\sigma_H.
 \end{split}
\end{equation}
Lemma~\ref{lem:globalflux} justifies \eqref{eq:rA}. Its absolute value is bounded by
\[
 \delta\int_{\Sigma}\rho^2(1+|H|)\,d\sigma_H
       \le C\delta(P_H+\EE).
\]
Integrating \eqref{eq:residualdensity} gives \eqref{eq:Rbound}. For fixed $\eps>0$, $\Phi_\eps\ge\eps$ and the fixed-surface bound on $\kappa$ make these integrals finite. The characteristic set is area-null, so the calculation on $\rho>0$ suffices.
\end{proof}

\subsection{The quantitative estimate}\label{sec:quantitativeestimate}
\begin{proof}[Proof of Proposition~\ref{prop:prescribedestimate}]
Write $P=P_H(F)$, $I_F=\int_\Sigma w_F\,d\sigma_{g_1}$ and
$\bar e_u=P^{-1}\int_\Sigma e_u\,d\sigma_H$. Equation~\eqref{eq:areaflux} and Cauchy--Schwarz give
\begin{equation}\label{eq:fluxbounds}
 2I_F=P(3-\bar e_u),\qquad
 I_F^2\le\frac{9P^2+P\EE}{2}.
\end{equation}
For $h=2+r$, the vertical-translation identity and $w_F^2-{N_T}w_K=\rho^4$ give
\[
 \mathcal B_h(w_F,w_F)=[\mathcal B_\eps]_*-\mathcal A_r,
 \qquad
 \mathcal A_r=\int_\Sigma[\rho^4Nr+w_KTr]\,d\sigma_{g_1}.
\]
Since $w_K$ is bounded and $w_F\ge1$,
$|\mathcal A_r|\le C\delta I_F\le C\delta(P+\EE)$.
Substituting \eqref{eq:fluxbounds} into \eqref{eq:integrated} and completing the square in $e_u$ gives the exact identity
\begin{equation}\label{eq:estimatecancellation}
 \begin{split}
 \mathcal B_h(w_F,w_F)+\frac4{3P}I_F^2
 ={}&-\EE-\int_\Sigma(e_u-\bar e_u)^2\,d\sigma_H
       -\frac{2P}{3}\bar e_u^2\\
 &+\int_\Sigma R_\eps\,d\sigma_H-\mathcal A_r.
 \end{split}
\end{equation}
By \eqref{eq:fluxbounds},
\[
 \left(c-\frac4{3P}\right)_+I_F^2
 \le\frac12\left(cP-\frac43\right)_+(9P+\EE).
\]
Using \eqref{eq:Rbound} and the bound on $\mathcal A_r$, adding $(c-4/(3P))I_F^2$, and using this inequality, we obtain from \eqref{eq:estimateform}
\[
 0\le\mathcal B_h(w_F,w_F)+cI_F^2
 \le-\EE-\NN_\eps+C_*q(P+\EE),
\]
where $C_*$ is an absolute constant.
For $q\le q_0\le\min\{1,(2C_*)^{-1}\}$, absorption gives
$\EE/2+\NN_\eps\le C_*qP$, proving \eqref{eq:robustCE}.
The integration in \eqref{eq:rA} keeps this estimate uniform in $\delta/\eps$.
\end{proof}

\section{Smooth stable approximations}\label{sec:selection}
Fix the bounded open density-one representative of an arbitrary horizontal minimizer $E$. Its boundary is Ahlfors regular of homogeneous dimension three and hence has Haar measure zero \cite{LR}. Normalize it as in Section~\ref{sec:prelim}, with $V=|E|$, $P=P_H(E)$ and $3P=8V$.
Fix a sufficiently large Euclidean ball $\mathcal O$ with $\overline E\Subset\mathcal O$ and a bounded neighbourhood $\mathcal U$ of $\overline{\mathcal O}$. All competitors below are subsets of $\mathcal O$, with full perimeter after extension by zero.

\begin{proposition}[Smooth stable approximation]\label{prop:stableapprox}
There are a fixed Euclidean ball $\mathcal B\Subset\mathcal O$, a sequence $\eps_j\downarrow0$, smooth regions $F_j\Subset\mathcal B$, and functions $r_j\in C^\infty(\mathcal U)$ such that $\partial F_j$ is stationary for
\[
 P_{\eps_j}(F)-\int_F(2+r_j)\,d\mathcal L^3.
\]
Writing $\EE_j=\EE(F_j)$ and $\NN_j=\NN_{\eps_j}(F_j)$, one has
\begin{equation}\label{eq:stableapprox}
 F_j\longrightarrow E\ \text{in }L^1,\qquad
 P_H(F_j)\longrightarrow P,\qquad
 \|r_j\|_{C^1(\mathcal U)}+\EE_j+\NN_j\longrightarrow0.
\end{equation}
Moreover, $c_j=-I''(|F_j|)\to1/(2V)$ and, for every smooth normal speed $v$ on $\Sigma_j=\partial F_j$,
\begin{equation}\label{eq:fullstability}
 \mathcal B_{2+r_j}(v,v)
 +c_j\left(\int_{\Sigma_j}v\,d\sigma_{g_1}\right)^2\ge0.
\end{equation}
\end{proposition}

Choose $\psi_E\in C_c^\infty(\HH)$ with $|\psi_E|\le1$, positive on $E$, negative on $\mathcal O\setminus\overline E$, and zero on $\partial E$. To construct it, sum positive bumps in $E$ and in $\HH\setminus\overline E$, with coefficients decreasing in every $C^k$ norm, subtract the second sum from the first, and multiply by a compactly supported cutoff equal to one on $\overline{\mathcal O}$. Set
\[
 \mathscr P_E(F)=\int_E\psi_E\,d\mathcal L^3-\int_F\psi_E\,d\mathcal L^3
 =\int_{F\triangle E}|\psi_E|\,d\mathcal L^3\ge0.
\]
For every $a>0$,
\begin{equation}\label{eq:pinL1}
 |F\triangle E|\le|\mathcal O\cap\{|\psi_E|\le a\}|
        +a^{-1}\mathscr P_E(F).
\end{equation}
The first term tends to zero as $a\downarrow0$, because the zero set of $\psi_E$ in $\mathcal O$ is $\partial E$.

Set $I(0)=0$. By \eqref{eq:profile},
$I(w)=P(w/V)^{3/4}$ for $w\ge0$. This identity uses only dilation invariance and $P=I(V)$, without the sharp value of $I(1)$. For $\lambda>0$, define
\begin{equation}\label{eq:penalty}
 \mathcal J_{\eps,\lambda}(F)=P_\eps(F)-I(|F|)+\lambda\mathscr P_E(F),
 \qquad F\subset\mathcal O.
\end{equation}
For each fixed $\eps>0$, $P_\eps$ supplies ellipticity for the regularity argument. The penalty $\mathscr P_E$ selects $E$: its convergence to zero implies $F\to E$ in $L^1$ by \eqref{eq:pinL1}. At a stationary set of positive volume, $-I(|F|)$ contributes
\[
 -I''(|F|)\left(\int_{\partial F}v\,d\sigma_{g_1}\right)^2
\]
to the full second variation, in addition to the prescribed-curvature terms. As $|F|\to V$, its coefficient tends to $4/(3P)$, the critical value in Proposition~\ref{prop:prescribedestimate}.
The pairings in Section~\ref{sec:rigidity} involve contact speeds whose volume derivatives need not vanish. Applying Cauchy--Schwarz to these pairings therefore uses nonnegativity on all normal speeds, as in \eqref{eq:fullstability}. The unconstrained minimization of $\mathcal J_{\eps,\lambda}$ retains precisely this full second variation, including the volume term above.
Here $P_0=P_H$. Since $P_\eps(F)\ge P_H(F)\ge I(|F|)$,
\begin{equation}\label{eq:penaltylower}
 \mathcal J_{\eps,\lambda}(F)
 \ge P_H(F)-I(|F|)+\lambda\mathscr P_E(F)\ge0.
\end{equation}

\begin{lemma}[Uniform density and confinement]\label{lem:selectionconfinement}
Fix $C>0$. For $0<\eps\le1$, suppose a set $F\subset\mathcal O$ of finite $g_1$-perimeter satisfies
\begin{equation}\label{eq:setcomparison}
 P_\eps(F)\le P_\eps(G)+C|F\triangle G|
\end{equation}
for every $G\subset\mathcal O$.
There are $s_0,c_d,d_0>0$ and a Euclidean ball $\mathcal B\Subset\mathcal O$, independent of $\eps,F$, such that, for every $x\in\supp D\mathbf1_F$ and $0<s\le s_0$,
\begin{equation}\label{eq:selectiondensity}
 |F\cap B_s^{g_1}(x)|\ge c_ds^4.
\end{equation}
If $|F\triangle E|<d_0$, then $F$ has a representative $F\Subset\mathcal B$.
\end{lemma}
\begin{proof}
Fix $x\in\supp D\mathbf1_F$, put $r(q)=\operatorname{dist}_{g_1}(x,q)$, $B_s=\{r<s\}$ and $m(s)=|F\cap B_s|$. Since $dV_{g_1}=d\mathcal L^3$ and $|\nabla_{g_1}r|_{g_1}=1$ almost everywhere, the coarea formula in local coordinates \cite[(12.6)--(12.7), pp.~69--70]{Simon} makes $m$ absolutely continuous, with $m'(s)=\int_{\partial B_s}\mathbf1_F\,d\sigma_{g_1}$ for almost every $s$, using the BV trace. At such radii the traces agree and the sphere carries no perimeter mass. The finite-perimeter cutting formula \cite[Proposition~4.3]{Simmons} and $\Phi_\eps(\nabla_{g_1}r)\le1$ then give
\begin{align*}
 P_\eps(F\setminus B_s)
 &\le P_\eps(F)-P_\eps(F;B_s)+m'(s),\\
 P_H(F\cap B_s)&\le P_\eps(F;B_s)+m'(s).
\end{align*}
Apply \eqref{eq:setcomparison} to deletion, and then apply nonsharp horizontal isoperimetry \cite{LR} with constant $c_{\mathrm{iso}}>0$ to the intersection, to obtain
\[
 P_\eps(F;B_s)\le m'(s)+Cm(s),\qquad
 c_{\mathrm{iso}}m(s)^{3/4}\le P_H(F\cap B_s)\le2m'(s)+Cm(s).
\]
Deletion is admissible even at $\partial\mathcal O$. Choose $s_0$ with $C\sup_{x\in\overline{\mathcal O}}|B_{s_0}^{g_1}(x)|^{1/4}\le c_{\mathrm{iso}}/2$. Then $Cm(s)\le(c_{\mathrm{iso}}/2)m(s)^{3/4}$, so $m'(s)\ge(c_{\mathrm{iso}}/4)m(s)^{3/4}$. Boundary support gives $m(a)>0$ for every $a>0$. Integrating the differential inequality from $a$ to $s\le s_0$ and letting $a\downarrow0$ gives $m(s)^{1/4}\ge(c_{\mathrm{iso}}/16)s$, since $m(a)\le|B_a|\to0$. This proves \eqref{eq:selectiondensity} with $c_d=(c_{\mathrm{iso}}/16)^4$, uniformly in $\eps$.

If $d=\operatorname{dist}_{g_1}(x,\overline E)>0$, then
$c_ds^4\le|F\triangle E|$ for $s<\min\{d,s_0\}$. For $|F\triangle E|<c_ds_0^4$ this gives
\[
 \operatorname{dist}_{g_1}(x,\overline E)
       \le(|F\triangle E|/c_d)^{1/4}.
\]
Shrink $d_0$ so this neighbourhood lies compactly in a fixed Euclidean ball $\mathcal B\Subset\mathcal O$, and choose a smaller ball $\mathcal B'\Subset\mathcal B$ still containing it. Then $D\mathbf1_F=0$ outside $\mathcal B'$; connectedness of that exterior and finite volume make $\mathbf1_F$ zero there. Thus $|F\setminus\mathcal B'|=0$, and $F$ has a representative $F\Subset\mathcal B$.
\end{proof}

\begin{lemma}[Elliptic regularity]\label{lem:fixedregularity}
For each fixed $\eps>0$, a set $F\Subset\mathcal O$ of finite $g_1$-perimeter satisfying \eqref{eq:setcomparison} has a $C^{1,\alpha}$ boundary representative for some $\alpha>0$. If its first variation prescribes smooth curvature, the boundary is smooth.
\end{lemma}
\begin{proof}
The reduced-boundary form of \eqref{eq:ellipticperimeter} is
\begin{equation}\label{eq:riemannianperimeter}
 P_\eps(F)=\eps\operatorname{Area}_{g_\eps}(\partial^*F)
         =\eps\int_{\partial^*F}d\sigma_{g_\eps}.
\end{equation}
For fixed $\eps>0$, $g_\eps$ is smooth and uniformly positive definite on a bounded neighbourhood of $\overline{\mathcal O}$. In coordinates, the normal integrand for $P_\eps$ is the square root of a quadratic form whose coefficient matrix is $\eps^2$ times the cofactor matrix of $g_\eps$; extend this matrix smoothly to the identity outside this neighbourhood. For competitors with $F\triangle G\Subset\mathcal O\cap B_r(x)$, locality and \eqref{eq:setcomparison} give an error at most $C|F\triangle G|\le C\omega_3r^3\le C\omega_3r^{5/2}$ for $r<1$. This is Simmons' almost-minimality condition with exponent $1/2$ \cite[Definition~2.2]{Simmons}.

Choose a representative with $\partial F=\supp|D\mathbf1_F|$. The two-phase density bounds \cite[Proposition~4.5]{Simmons} make the density-one representative open with the same boundary. Theorem~1.1 of \cite{Simmons} gives $C^{1,1/8}$ regularity and no singular points in dimension three.

In a rotated graph chart write $z=w(y)$, and let $\Psi(y,z,p)$ be the smooth graph integrand for $P_\eps$. Positive definiteness of $g_\eps$ makes $\Psi_{pp}$ positive definite, uniformly on compact jet sets. For $F$ below the graph, stationarity with prescribed smooth curvature $h$ gives
\[
 \operatorname{div}_y\Psi_p(y,w,Dw)
       =\Psi_z(y,w,Dw)-h(y,w).
\]
The sign of $h$ is reversed for $F$ above the graph. Rosales' interior estimate gives $w\in C^2$ \cite[Lemma~2.4, pp.~610--611]{Rosales}. Interior Schauder estimates \cite[Chapter~6]{GilbargTrudinger} then give $w\in C^{2,1/8}$; differentiating the uniformly elliptic equation and iterating yields smoothness.
\end{proof}

\begin{proof}[Proof of Proposition~\ref{prop:stableapprox}]
We construct the approximations by minimizing \eqref{eq:penalty}. To force convergence to $E$, we first construct smooth comparison regions with energy small relative to the penalty coefficient. Choose a nonnegative $\varphi\in C_c^\infty(B_1(0))$ with $\int_{\mathbb R^3}\varphi\,d\mathcal L^3=1$, where $B_1(0)$ is the Euclidean unit ball. Extend $\mathbf1_E$ by zero and define its Euclidean mollifications by
\[
 v_l(\xi)=\int_{\mathbb R^3}l^3\varphi\bigl(l(\xi-\zeta)\bigr)
                  \mathbf1_E(\zeta)\,d\mathcal L^3(\zeta).
\]
Since $\overline E\Subset\mathcal O$, after discarding finitely many terms, $v_l\in C_c^\infty(\mathcal O)$, $0\le v_l\le1$, and all supports lie in a fixed compact subset of $\mathcal O$. Joint horizontal BV approximation \cite[Theorem~2.4 and its proof, pp.~495--496]{AGM} gives
\[
 \|v_l-\mathbf1_E\|_{L^1}\longrightarrow0,\qquad
 \int_{\mathcal O}\sqrt{(Xv_l)^2+(Yv_l)^2}\,d\mathcal L^3
       \longrightarrow P.
\]
Choose $0<a_l<1/2$ with $a_l\downarrow0$ and $\|v_l-\mathbf1_E\|_{L^1}/a_l\to0$. The weighted coarea formula gives
\[
 \int_0^1 P_H(\{v_l>t\})\,dt
   =\int_{\mathcal O}\sqrt{(Xv_l)^2+(Yv_l)^2}\,d\mathcal L^3.
\]
By averaging over $(a_l,1-a_l)$ and using Sard's theorem, choose a regular value $t_l$ in this interval such that $G_l:=\{v_l>t_l\}$ satisfies
\[
 P_H(G_l)\le\frac1{1-2a_l}
       \int_{\mathcal O}\sqrt{(Xv_l)^2+(Yv_l)^2}\,d\mathcal L^3,
 \qquad
 |G_l\triangle E|\le\frac{\|v_l-\mathbf1_E\|_{L^1}}{a_l}.
\]
Thus $G_l$ has smooth boundary and common compact support in $\mathcal O$, and $G_l\to E$ in $L^1$. The first bound gives $\limsup_lP_H(G_l)\le P$, while lower semicontinuity gives the reverse inequality for the lower limit. Hence $P_H(G_l)\to P$. Each $G_l$ also has finite $P_1$.

We choose the parameters so that the comparison energy is at most $\lambda_j^2$. The lower bound \eqref{eq:penaltylower} will then give $\mathscr P_E(F_j)\le\lambda_j$ for a minimizer. The preceding convergences and the boundedness of $\psi_E$ give $\mathcal J_{0,\lambda}(G_l)\to0$ for each fixed $\lambda>0$. For $\lambda_j\downarrow0$ with $\lambda_j\le1$, select a subsequence $G_j$ with $\mathcal J_{0,\lambda_j}(G_j)\le\lambda_j^2/2$. Choose $0<\eps_j\le\lambda_j$ with $\eps_jP_1(G_j)\le\lambda_j^2/2$ and $\eps_j\le\eps_{j-1}/2$. Since $P_{\eps_j}(G_j)\le P_H(G_j)+\eps_jP_1(G_j)$, one has $\mathcal J_{\eps_j,\lambda_j}(G_j)\le\lambda_j^2$.

Minimize $\mathcal J_{\eps_j,\lambda_j}$ at each fixed $j$. Bounded energy bounds $P_1$, since the volume terms are bounded and continuous on $[0,|\mathcal O|]$ and $P_{\eps_j}\ge\eps_jP_1$. BV compactness and lower semicontinuity \cite[Theorem~3.1 and Proposition~3.1]{Simmons}, together with continuity of the penalties, give a minimizer $F_j\subset\mathcal O$; zero extension and $|\partial\mathcal O|=0$ preserve admissibility. Minimality and \eqref{eq:penaltylower} give
\[
 0\le P_H(F_j)-I(|F_j|)+\lambda_j\mathscr P_E(F_j)
 \le\mathcal J_{\eps_j,\lambda_j}(F_j)\le\lambda_j^2.
\]
In particular $\mathscr P_E(F_j)\le\lambda_j$. Equation~\eqref{eq:pinL1}, first with $j\to\infty$ and then $a\downarrow0$, gives $F_j\to E$ in $L^1$. Hence $w_j:=|F_j|\to V$, and $w_j\ge V/2$ for all sufficiently large $j$. Also
\[
 P_{\eps_j}(F_j)\le I(w_j)+\lambda_j^2\le I(|\mathcal O|)+1.
\]

To obtain common compact support inside $\mathcal O$ from $F_j\to E$, we verify \eqref{eq:setcomparison} with a constant independent of $j$ and apply Lemma~\ref{lem:selectionconfinement}. Interior regularity will then apply on this fixed region. On $[V/4,|\mathcal O|]$, the profile $I$ is Lipschitz, and $\mathscr P_E$ changes by at most $|F_j\triangle G|$. Minimality therefore gives \eqref{eq:setcomparison} for $F_j$ whenever $|G|\ge V/4$, with a constant independent of $j$. If $|G|<V/4$, then $|F_j\triangle G|\ge V/4$, so the perimeter bound gives the same comparison after increasing the constant. Since $|F_j\triangle E|\to0$, Lemma~\ref{lem:selectionconfinement} places all sufficiently late $F_j$ compactly in a fixed ball $\mathcal B\Subset\mathcal O$.

Lemma~\ref{lem:fixedregularity} gives a $C^{1,1/8}$ representative. The unconstrained first variation of \eqref{eq:penalty} prescribes smooth curvature
\[
 H_{\eps_j}=2+r_j,\qquad
 r_j=I'(w_j)-2+\lambda_j\psi_E,
\]
and the same lemma then gives smoothness. Since $I'(V)=2$,
\[
 \|r_j\|_{C^1(\mathcal U)}\le C(|F_j\triangle E|+\lambda_j)\longrightarrow0.
\]

The unconstrained minimum also supplies the full second-variation inequality. With $r_j$ fixed, differentiating \eqref{eq:penalty} twice at $F_j$ gives \eqref{eq:fullstability} for every smooth normal speed, with $c_j=-I''(w_j)>0$. Proposition~\ref{prop:prescribedestimate} controls the remaining geometric errors once these coefficients are compared with the critical value for $P_j=P_H(F_j)$.

The energy bound gives $0\le P_H(F_j)-I(w_j)\le\lambda_j^2$, and hence $P_j\to P$. Since $-I''(V)=1/(2V)=4/(3P)$,
\[
 c_j\longrightarrow\frac1{2V},\qquad
 \left(c_jP_j-\frac43\right)_+\longrightarrow0.
\]
Proposition~\ref{prop:prescribedestimate}, applied with $c=c_j$, gives $\EE_j+\NN_j\to0$, proving \eqref{eq:stableapprox}.
\end{proof}

\section{Centre rigidity and sharp isoperimetry}\label{sec:rigidity}
For the approximations of Proposition~\ref{prop:stableapprox}, the estimates of Section~\ref{sec:prelim} control the variation of the reconstructed centre along characteristic curves away from characteristic points. The full second variation forces these centres to concentrate at a single point, and Proposition~\ref{lem:stablelimit} identifies the limiting region as a left translate of $\Omega$.

In every limit involving $j$ and a cutoff width $\eta$, $\eta>0$ is fixed while $j\to\infty$, and $\eta\downarrow0$ is taken afterwards. Constants may depend on the fixed cutoff; bounds uniform in $\eta$ are stated explicitly.

\subsection{Angular averages and contact variations}

\begin{lemma}[Angular averages]\label{lem:perimeterlimit}
Let $\Sigma$ be a smooth closed oriented surface whose reconstructed centres $\mathbf a=(z_0,t_0)$ lie in a compact set $K\subset\HH$. Put $P=\sigma_H(\Sigma)$. For $\varphi\in C^\infty(\HH)$ and $f\in C(\bar{\angI})$, set $c_f=(2/\pi)\int_{\angI}\rho^2f\,du$. Then
\begin{equation}\label{eq:angleaverage}
 \left|\int_\Sigma\varphi(\mathbf a)(f(u)-c_f)\,d\sigma_H\right|
 \le C\|\varphi\|_{C^1(K)}\|f\|_\infty\sqrt{P\EE}.
\end{equation}
Here $C$ is absolute and derivatives in the $C^1$ norm are measured using $g_1$. The characteristic set has zero $g_1$-area.
\end{lemma}
\begin{proof}
Equations \eqref{eq:e}, \eqref{eq:coercivity} and \eqref{eq:centreDZ} give, on the noncharacteristic set,
\[
 |Z(\varphi(\mathbf a))|
 \le C\|\varphi\|_{C^1(K)}(|Zz_0|+\rho^2|e_H|)
 \le C\|\varphi\|_{C^1(K)}\sqrt e.
\]
For smooth $f$, put
\[
 \chi(u)=\rho(u)^{-2}
       \int_{-\pi/2}^{u}\rho(v)^2(f(v)-c_f)\,dv.
\]
The numerator has zero total integral. Writing $d$ for the distance to the nearer endpoint gives $|\chi(u)|\le d\|f-c_f\|_\infty$ and $|2\tan u\,\chi(u)|\le2\|f-c_f\|_\infty$. Hence $\chi=O(\rho)$ at both endpoints and
\[
 \chi'-2\tan u\,\chi=f-c_f,\qquad
 \|\chi'\|_\infty\le6\|f\|_\infty,\qquad
 \|\chi\|_\infty\le\pi\|f\|_\infty.
\]
Cut $\chi$ off within distance $\tau$ of the endpoints and integrate the divergence of $\rho\varphi(\mathbf a)\chi(u)Z$. The cut-off fields are smooth and vanish near the characteristic set. The additional cutoff derivatives are uniformly bounded because $\chi=O(\rho)$; the divergence terms, written with respect to $\sigma_H$, are dominated by a constant times $1+\sqrt e$. Thus dominated convergence as $\tau\downarrow0$, using $Zu=-1+e_u$ and $\Div_\Sigma(\rho Z)=2N_T$, gives
\[
 \int_\Sigma\varphi(\mathbf a)(f(u)-c_f)\,d\sigma_H
 =\int_\Sigma[\varphi(\mathbf a)e_u\chi'(u)
        +\chi(u)Z(\varphi(\mathbf a))]\,d\sigma_H.
\]
Cauchy--Schwarz proves \eqref{eq:angleaverage}. Uniform approximation extends the estimate to continuous $f$.
\end{proof}

For a smooth real function $h$ on $\HH$, the contact Hamiltonian construction \cite[Proposition~5.3, p.~227]{Ovsienko} gives, in our normalization,
\begin{equation}\label{eq:contactfield}
 V_h=-\tfrac12(Yh)X+\tfrac12(Xh)Y+hT,
 \qquad \mathcal L_{V_h}\Theta=(Th)\Theta.
\end{equation}
The second identity follows from $d\Theta=2\,dx\wedge dy$ and Cartan's formula. The flow preserves the horizontal distribution, which yields the $\eps$-uniform estimate in Lemma~\ref{lem:contactbound}.

\begin{lemma}[Uniform second-variation bound under contact flows]\label{lem:contactbound}
Fix a compact set $K\Subset\mathcal U$ and $h\in C_c^\infty(\HH)$. Let $0<\eps\le1$ and let $F\subset K$ be a smooth region stationary for $P_\eps-\int_F(2+r)\,d\mathcal L^3$, where $\|r\|_{C^1(\mathcal U)}\le1$. For $v=\langle V_h,N\rangle$ on $\Sigma=\partial F$, the Hessian in Lemma~\ref{lem:killinghessian} satisfies
\begin{equation}\label{eq:contactbound}
 |\mathcal B_{2+r}(v,v)|\le C_h\operatorname{Area}_{g_1}(\Sigma).
\end{equation}
The constant is independent of $\eps$ and $F$.
\end{lemma}
\begin{proof}
Let $\Psi_s$ be the flow of $V_h$. In the oriented orthonormal frames $(X,Y,T)$ at the source and target, write $M_s$ for its differential. Since $\Psi_s$ preserves the horizontal distribution, $M_s$ is upper block triangular. Put $Q_\eps=\operatorname{diag}(1,1,\eps)$ and $\operatorname{cof}M=(\det M)M^{-T}$. Then
\[
 C_{\eps,s}:=Q_\eps\operatorname{cof}(M_s)Q_\eps^{-1}
 =\begin{pmatrix}*&0\\\eps *&*\end{pmatrix},
\]
where each asterisk denotes a smooth block independent of $\eps$. For $p$ in a fixed neighbourhood of $K$ and sufficiently small $s$, this matrix, its inverse, and its first two $s$-derivatives are uniformly bounded for $0<\eps\le1$. The area transformation formula gives
\[
 P_\eps(\Psi_sF)
 =\int_\Sigma
       |Q_\eps\operatorname{cof}(M_s)N|\,d\sigma_{g_1}
 =\int_\Sigma|C_{\eps,s}Q_\eps N|\,d\sigma_{g_1}.
\]
The matrix bounds imply
\[
 \left|\frac{d^2}{ds^2}|C_{\eps,s}Q_\eps N|\right|
 \le C_h|Q_\eps N|=C_h\Phi_\eps.
\]
Transport of the volume integral gives
\[
 \left.\frac{d^2}{ds^2}\right|_{s=0}
       \int_{\Psi_sF}(2+r)\,d\mathcal L^3
 =\int_\Sigma\operatorname{div}((2+r)V_h)\,v\,d\sigma_{g_1}.
\]
Since $\operatorname{div}((2+r)V_h)=V_hr+(2+r)\operatorname{div}V_h$, the $C^1$ bound on $r$ bounds this term by $C_h\operatorname{Area}_{g_1}(\Sigma)$. The perimeter term has the same bound because $\Phi_\eps\le1$. Stationarity identifies the second derivative along $\Psi_s$ with the Hessian on its normal speed, proving \eqref{eq:contactbound}.
\end{proof}

\subsection{Zero mode and endpoint pairings}

The function $\rho(u)=\cos u$ vanishes at the endpoints of $\angI$ and satisfies $\rho''+\rho=0$. It is therefore a zero mode of the one-dimensional form
\[
 \int_{\angI}(k'^2-k^2)\,du,\qquad k\in H_0^1(\angI).
\]
Cauchy--Schwarz gives $\mathcal Q(w,w)=0\Rightarrow\mathcal Q(w,v)=0$ for every $v$ whenever $\mathcal Q$ is a nonnegative symmetric bilinear form. The cutoff construction below applies the approximate form of this implication.

At the angular endpoints of the approximations, the normal speed of the contact variation generated by $h$ approaches $h(\mathbf a^+)$ and $-h(\mathbf a^-)$, where $\mathbf a^\pm=(x_0,y_0,t_0\pm\pi/4)$. Lemma~\ref{lem:endpointpairings} makes this endpoint relation precise.

For speeds $k(u_j)\varphi_j(\mathbf a_j)$, the characteristic derivatives give the one-dimensional form above in the limit, and the mixed pairing with a contact speed contains the corresponding operator $-k''-k$. Cutoffs of $\rho$ have limiting energy $O(\eta)$ when the centre test has zero horizontal-perimeter average; this zero average also makes the volume term tend to zero for each fixed cutoff. Cauchy--Schwarz for the full second variation then bounds the mixed pairings by $O(\sqrt\eta)$.

For these cutoffs, $k''+k$ is supported near the two angle endpoints. The resulting difference $h(\mathbf a_j^+)-h(\mathbf a_j^-)$ turns the pairings into coordinate variances when $h$ agrees with $xt$, $yt$ or $t^2$ near the endpoints. Thus the zero mode and endpoint values together force the reconstructed centres to concentrate.

Let $\eps_j\downarrow0$ and let $F_j$ be smooth regions in a common compact set $K\Subset\mathcal U$, stationary for
\[
 P_{\eps_j}(G)-\int_G(2+r_j)\,d\mathcal L^3,
 \qquad \|r_j\|_{C^1(\mathcal U)}\longrightarrow0.
\]
Suppose $\sup_jP_H(F_j)<\infty$ and $\EE_j+\NN_{\eps_j}(F_j)\to0$.
Write $\Sigma_j=\partial F_j$, let $N_j$ be its outward $g_1$-unit normal, and put $P_j=P_H(F_j)$ and $\mathcal B_j=\mathcal B_{2+r_j}$.
Choose a compact Euclidean ball $K_0$ containing all reconstructed centres, and let $\varphi_j\in C^\infty(\HH)$ satisfy $\sup_j\|\varphi_j\|_{C^1(K_0)}<\infty$.

\begin{lemma}[Second-variation pairings]\label{lem:localpairings}
For $F_j$ and $\varphi_j$ as above, fix $h\in C_c^\infty(\HH)$ and $k\in C_c^\infty(\angI)$, and put
\[
 w_j=k(u_j)\varphi_j(\mathbf a_j),\qquad v_j=\langle V_h,N_j\rangle,
\]
extending $w_j$ by zero near the characteristic set. Then, as $j\to\infty$,
\begin{equation}\label{eq:mixedcontact}
 \mathcal B_j(w_j,v_j)
 =-\int_{\Sigma_j}\rho_j^{-1}(k''+k)(u_j)
          \varphi_j(\mathbf a_j)v_j\,d\sigma_{g_1}+o_j(1).
\end{equation}
\begin{equation}\label{eq:localformlimit}
 \mathcal B_j(w_j,w_j)
 =\frac2\pi\left(\int_{\Sigma_j}\varphi_j(\mathbf a_j)^2\,d\sigma_{H,j}\right)
       \int_{\angI}(k'^2-k^2)\,du+o_j(1).
\end{equation}
\end{lemma}

\begin{proof}
Put $\Phi_j=\Phi_{\eps_j}$ and $\NN_j=\NN_{\eps_j}(F_j)$. The error constants depend on the fixed functions $k,h$, the common compact sets, the common perimeter bound, and a common $C^1(K_0)$ bound for $\varphi_j$; they are independent of $j$ and may increase as $\supp k$ approaches the angle endpoints.

To control the transverse contributions on the fixed angular support of $k$, first note from \eqref{eq:areabound} that
\[
 \operatorname{Area}_{g_1}(\Sigma_j)
 \le\tfrac12(3P_j+\sqrt{P_j\EE_j})\le C.
\]
Fix $\vartheta>0$ such that $\rho\ge\vartheta$ on $\supp k$. On this fixed angle interval, \eqref{eq:coercivity} controls the characteristic errors and $\NN_j$ controls the transverse terms. The chain rule gives
\[
 \begin{aligned}
 Z_jw_j&=k'(u_j)(-1+e_{u,j})\varphi_j(\mathbf a_j)
       +k(u_j)Z_j(\varphi_j(\mathbf a_j)),\\
 U_jw_j&=k'(u_j)\kappa_j\varphi_j(\mathbf a_j)
       +k(u_j)U_j(\varphi_j(\mathbf a_j)),\\
 \xi v_j&=\langle\nabla_\xi V_h,N_j\rangle
       +\langle V_h,A\xi\rangle,
       \qquad \xi\in\{Z_j,U_j\}.
 \end{aligned}
\]
On this angle interval, equations \eqref{eq:centreD} give
$|Z_j(\varphi_j(\mathbf a_j))|\le C(|e_{H,j}|+|e_{u,j}|)$ and
$|U_j(\varphi_j(\mathbf a_j))|\le C_\vartheta(1+|\kappa_j|+|e_{u,j}|)$.
Together with $AZ=\rho H Z+e_uU$ and $AU=e_uZ+\kappa U$, these identities imply
\[
 \begin{aligned}
 Z_jw_j&=-k'(u_j)\varphi_j(\mathbf a_j)
          +O_k(|e_{H,j}|+|e_{u,j}|),\\
 |U_jw_j|+|U_jv_j|&\le C_{k,h}(1+|\kappa_j|+|e_{u,j}|),\\
 |v_j|&\le C_h,\qquad
 |Z_jv_j|\le C_h(1+|e_{H,j}|+|e_{u,j}|).
 \end{aligned}
\]
The coercivity estimate \eqref{eq:coercivity} and $d\sigma_{H,j}=\rho_jd\sigma_{g_1}$ imply
\begin{equation}\label{eq:localerrorbounds}
 \begin{aligned}
 \int_{\Sigma_j\cap\{\rho_j\ge\vartheta\}}
       (e_{H,j}^2+e_{u,j}^2)\,d\sigma_{g_1}&\le C_\vartheta\EE_j,\\
 \int_{\Sigma_j\cap\{\rho_j\ge\vartheta\}}
       (|e_{H,j}|+|e_{u,j}|)\,d\sigma_{g_1}&\le C_\vartheta\sqrt{\EE_j},
 \end{aligned}
\end{equation}
where the second line uses Cauchy--Schwarz and the area bound \eqref{eq:areabound}. By the definition of $\NN_j$ in \eqref{eq:e} and $d\sigma_{H,j}=\rho_jd\sigma_{g_1}$, the transverse curvature term satisfies
\begin{equation}\label{eq:transversecurvature}
 \begin{aligned}
 \eps_j^2\int_{\Sigma_j\cap\{\rho_j\ge\vartheta\}}
             \frac{\kappa_j^2}{\Phi_j^3}\,d\sigma_{g_1}
 &=\eps_j^2\int_{\Sigma_j\cap\{\rho_j\ge\vartheta\}}
       \frac1{\rho_j\Phi_j}\frac{\kappa_j^2}{\Phi_j^2}\,d\sigma_{H,j}\\
 &\le\vartheta^{-2}\eps_j^2\int_{\Sigma_j}
       \frac{\kappa_j^2}{\Phi_j^2}\,d\sigma_{H,j}
 =\vartheta^{-2}\NN_j\longrightarrow0.
 \end{aligned}
\end{equation}
Here $\Phi_j\ge\rho_j\ge\vartheta$, and $\vartheta$ is fixed. Combining \eqref{eq:transversecurvature} with \eqref{eq:localerrorbounds} and \eqref{eq:areabound} gives
\begin{equation}\label{eq:transverseerror}
 \int_{\Sigma_j\cap\{\rho_j\ge\vartheta\}}
       \frac{\eps_j^2}{\Phi_j^3}
       (1+\kappa_j^2+e_{u,j}^2)\,d\sigma_{g_1}
 \le C_\vartheta\bigl(\NN_j+\eps_j^2(1+\EE_j)\bigr).
\end{equation}
Equation~\eqref{eq:transverseerror} makes the transverse contributions, weighted by $\eps_j^2/\Phi_j^3$, tend to zero.

Here $N_jr_j=dr_j(N_j)=\langle\nabla_{g_1}r_j,N_j\rangle_{g_1}$ is the ambient normal derivative on $\Sigma_j$. Lemma~\ref{lem:killinghessian}, applied with $h=2+r_j$, adds the potential term $-(N_jr_j)wv$ to \eqref{eq:hessian}, giving
\[
 \begin{aligned}
 \mathcal B_j(w,v)=\int_{\Sigma_j}\biggl[&
 \frac{Z_jw\,Z_jv}{\Phi_j}
 +\frac{\eps_j^2}{\Phi_j^3}U_jw\,U_jv\\
 &-\biggl(\frac{\rho_j^2(H_j^2-4)+(e_{u,j}+1)^2}{\Phi_j}\\
 &\qquad+\frac{\eps_j^2}{\Phi_j^3}
       (\kappa_j^2+(e_{u,j}-1)^2)+N_jr_j\biggr)wv
 \biggr]d\sigma_{g_1}.
 \end{aligned}
\]
On the support of $k$,
\[
 0\le\rho_j^{-1}-\Phi_j^{-1}
 =\frac{\eps_j^2\sin^2u_j}{\rho_j\Phi_j(\rho_j+\Phi_j)}
 \le C_\vartheta\eps_j^2.
\]
Moreover, $H_j=2+e_{H,j}$ gives
\[
 \begin{aligned}
 \rho_j^2(H_j^2-4)+(e_{u,j}+1)^2-1
 &=\rho_j^2(4e_{H,j}+e_{H,j}^2)+2e_{u,j}+e_{u,j}^2,\\
 \int_{\Sigma_j\cap\{\rho_j\ge\vartheta\}}
 \left|\rho_j^2(H_j^2-4)+(e_{u,j}+1)^2-1\right|d\sigma_{g_1}
 &\le C_\vartheta(\sqrt{\EE_j}+\EE_j).
 \end{aligned}
\]
The potential contribution is bounded by
$C_{k,h}\|r_j\|_{C^1(\mathcal U)}\operatorname{Area}_{g_1}(\Sigma_j)$.
Substituting the preceding bounds and derivative estimates into the Hessian reduces the mixed pairing to
\[
 \mathcal B_j(w_j,v_j)
 =-\int_{\Sigma_j}\rho_j^{-1}\varphi_j(\mathbf a_j)
       [k'(u_j)Z_jv_j+k(u_j)v_j]\,d\sigma_{g_1}+o_j(1).
\]
For fixed $k,h$, uniformly in $j$, the error is bounded by a constant times
$\sqrt{\EE_j}+\EE_j+\NN_j+\eps_j^2(1+\EE_j)+\|r_j\|_{C^1(\mathcal U)}$.

The remaining characteristic derivative of $v_j$ can be transferred to $k'$ by integration by parts. Since
\[
 \Div_{\Sigma_j}(\rho_j^{-1}Z_j)
       =2\sin u_j\,\rho_j^{-2}e_{u,j},
\]
integration by parts, using $Z_ju_j=-1+e_{u,j}$ and the centre derivative bounds, gives
\[
 \begin{aligned}
 -\int_{\Sigma_j}\rho_j^{-1}\varphi_j(\mathbf a_j)k'(u_j)Z_jv_j\,d\sigma_{g_1}
 &=\int_{\Sigma_j}v_j\Div_{\Sigma_j}
       (\rho_j^{-1}\varphi_j(\mathbf a_j)k'(u_j)Z_j)\,d\sigma_{g_1}\\
 &=-\int_{\Sigma_j}\rho_j^{-1}\varphi_j(\mathbf a_j)k''(u_j)v_j\,d\sigma_{g_1}
       +O_{k,h}(\sqrt{\EE_j}).
 \end{aligned}
\]
Combining this with the remaining $k(u_j)v_j$ term proves \eqref{eq:mixedcontact}.
All fields in this integration vanish near the characteristic set.

For the self-pairing, the same estimates leave an angular integral that Lemma~\ref{lem:perimeterlimit} evaluates. They give
\[
 \begin{aligned}
 \mathcal B_j(w_j,w_j)
 &=\int_{\Sigma_j}\rho_j^{-1}\varphi_j(\mathbf a_j)^2
       (k'^2-k^2)(u_j)\,d\sigma_{g_1}+o_j(1)\\
 &=\int_{\Sigma_j}\varphi_j(\mathbf a_j)^2
       \rho_j^{-2}(k'^2-k^2)(u_j)\,d\sigma_{H,j}+o_j(1).
 \end{aligned}
\]
Since $\rho_j\ge\vartheta>0$ on the support of the integrands, we may substitute $d\sigma_{g_1}=\rho_j^{-1}d\sigma_{H,j}$. The angular average now evaluates this self-pairing. For $f=\rho^{-2}(k'^2-k^2)$, extended by zero at the angle endpoints, equation~\eqref{eq:angleaverage} has error at most
$C_k\sqrt{P_j\EE_j}$, since the $C^1(K_0)$ norms of $\varphi_j^2$ are uniformly bounded.
Its angular integral is $\int_{\angI}\rho^2f\,du=\int_{\angI}(k'^2-k^2)\,du$, proving \eqref{eq:localformlimit}.
\end{proof}

\begin{samepage}
\begin{lemma}[Cutoff estimates for the second variation]\label{lem:cutoffpairings}
Let $F_j$, $K_0$, $\varphi_j$ and $h$ be as in Lemma~\ref{lem:localpairings}, and suppose
$\int_{\Sigma_j}\varphi_j(\mathbf a_j)\,d\sigma_{H,j}=0$.
For coefficients $c_j\ge0$ with $\sup_jc_j<\infty$, assume that the form
\begin{equation}\label{eq:fullform}
 \mathcal Q_j(v,w)=\mathcal B_j(v,w)
 +c_j\left(\int_{\Sigma_j}v\,d\sigma_{g_1}\right)
    \left(\int_{\Sigma_j}w\,d\sigma_{g_1}\right)
\end{equation}
is nonnegative on all smooth normal speeds.
For $0<\eta<\pi/4$, choose $\chi_\eta\in C_c^\infty(\angI)$ with $0\le\chi_\eta\le1$, vanishing within distance $\eta$ of the endpoints, equal to one at distance at least $2\eta$, and satisfying
$|\chi_\eta'|\le C/\eta$ and $|\chi_\eta''|\le C/\eta^2$, with $C$ independent of $\eta$.
Put $k_\eta=\chi_\eta\rho$ and $w_{j,\eta}=k_\eta(u_j)\varphi_j(\mathbf a_j)$.
Then, for fixed $\eta$,
\begin{equation}\label{eq:nullcutoff}
 0\le\limsup_j\mathcal Q_j(w_{j,\eta},w_{j,\eta})
 \le C\eta.
\end{equation}
\begin{equation}\label{eq:mixedcutoff}
 \limsup_j|\mathcal B_j(w_{j,\eta},v_j)|\le C_h\sqrt\eta,
\end{equation}
where the constants are independent of $\eta$.
\end{lemma}
\end{samepage}

\begin{proof}
For $k\in C_c^\infty(\angI)$, integration by parts and $\rho''+\rho=0$ give
\[
 \int_{\angI}(k'^2-k^2)\,du
 =\int_{\angI}\rho^2\left[\left(\frac{k}{\rho}\right)'\right]^2du.
\]
Lemma~\ref{lem:perimeterlimit} with $f=\chi_\eta$ gives
\[
 \left|\int_{\Sigma_j}w_{j,\eta}\,d\sigma_{g_1}\right|
 =\left|\int_{\Sigma_j}\chi_\eta(u_j)\varphi_j(\mathbf a_j)\,d\sigma_{H,j}\right|
 \le C\sqrt{P_j\EE_j}\longrightarrow0.
\]
Since the $c_j$ are bounded, the volume term in \eqref{eq:fullform} tends to zero. For $k_\eta=\chi_\eta\rho$, the zero-mode identity gives
\[
 \begin{aligned}
 \int_{\angI}(k_\eta'^2-k_\eta^2)\,du
 &=\int_{\angI}\rho^2|\chi_\eta'|^2\,du\\
 &\le C\eta^{-2}\int_{\{u\in\angI:\,\pi/2-|u|\le2\eta\}}
       (\pi/2-|u|)^2\,du\le C\eta.
 \end{aligned}
\]
Here $\rho(u)\le\pi/2-|u|$ near either endpoint. Also
$\int_{\Sigma_j}\varphi_j(\mathbf a_j)^2\,d\sigma_{H,j}\le C P_j\le C$.
Thus, for fixed $\eta$, equation~\eqref{eq:localformlimit} proves \eqref{eq:nullcutoff}.
The divergence theorem gives
\[
 \left|\int_{\Sigma_j}v_j\,d\sigma_{g_1}\right|
 =\left|\int_{F_j}\operatorname{div}V_h\,d\mathcal L^3\right|\le C_h.
\]
Together with Lemma~\ref{lem:contactbound} and boundedness of $c_j$, this gives $0\le\mathcal Q_j(v_j,v_j)\le C_h$, independently of $\eta$. Cauchy--Schwarz for $\mathcal Q_j$, and the preceding area-mean bounds, imply
\[
 \begin{aligned}[b]
 \limsup_j|\mathcal B_j(w_{j,\eta},v_j)|
 &=\limsup_j|\mathcal Q_j(w_{j,\eta},v_j)|\\
 &\le\limsup_j\sqrt{
       \mathcal Q_j(w_{j,\eta},w_{j,\eta})\,
       \mathcal Q_j(v_j,v_j)}\\
 &\le\sqrt{
       \bigl(\limsup_j\mathcal Q_j(w_{j,\eta},w_{j,\eta})\bigr)
       \bigl(\limsup_j\mathcal Q_j(v_j,v_j)\bigr)}
 \le C_h\sqrt\eta.
 \end{aligned}\qedhere
\]
\end{proof}

\begin{lemma}[Endpoint pairings]\label{lem:endpointpairings}
Let $F_j$, $K_0$, $\varphi_j$ and $h$ be as in Lemma~\ref{lem:localpairings},
and let $k_\eta$ and $w_{j,\eta}$ be the cutoff functions and normal speeds defined in Lemma~\ref{lem:cutoffpairings}.
Put $\mathbf a^\pm=(x_0,y_0,t_0\pm\pi/4)$. For fixed $0<\eta<\pi/4$,
\begin{equation}\label{eq:endpointapprox}
 \limsup_j\left|\mathcal B_j(w_{j,\eta},v_j)
       +\frac2\pi\int_{\Sigma_j}\varphi_j(\mathbf a_j)
           [h(\mathbf a_j^+)-h(\mathbf a_j^-)]\,d\sigma_{H,j}\right|
       \le C_h\eta,
\end{equation}
with $C_h$ independent of $\eta$.
\end{lemma}

\begin{proof}
The endpoint speeds follow from \eqref{eq:incidence}: $p=\mathbf a_j^\pm+O(\eta)$ for $0<\pi/2\mp u_j<2\eta$, since the centres lie in a common compact set. Together with $N_j=\rho_j\nu_j+\sin u_j\,T$, this yields, uniformly for $|\nu_j|=1$,
\[
 v_j=\begin{cases}
 h(\mathbf a_j^+)+O_h(\eta),&\pi/2-2\eta<u_j<\pi/2,\\
 -h(\mathbf a_j^-)+O_h(\eta),&-\pi/2<u_j<-\pi/2+2\eta.
 \end{cases}
\]
The function $k_\eta''+k_\eta=\rho\chi_\eta''+2\rho'\chi_\eta'$ is supported in these two intervals and satisfies
\[
 \begin{gathered}
 \int_{\angI}|k_\eta''+k_\eta|\,du\le C,\\
 \int_{-\pi/2}^{-\pi/2+2\eta}(k_\eta''+k_\eta)\,du=1+O(\eta^2),\qquad
 \int_{\pi/2-2\eta}^{\pi/2}(k_\eta''+k_\eta)\,du=1+O(\eta^2).
 \end{gathered}
\]
On either transition interval, $\rho\le2\eta$, so
$|k_\eta''+k_\eta|\le C\eta^{-1}$; the two intervals have total length at most $4\eta$.
This proves the first bound. For the upper endpoint,
\[
 \int_{\pi/2-2\eta}^{\pi/2}(k_\eta''+k_\eta)\,du
 =[k_\eta']_{\pi/2-2\eta}^{\pi/2}
       +\int_{\pi/2-2\eta}^{\pi/2}k_\eta\,du
 =\cos(2\eta)+O(\eta^2)=1+O(\eta^2).
\]
The lower endpoint is identical, with $k_\eta'=\cos(2\eta)$ at its inner edge and zero near the endpoint.
To average these endpoint values, define the two angle weights by
\[
 f_\eta^\pm(u)=\begin{cases}
 \rho(u)^{-2}(k_\eta''+k_\eta)(u),&0<\pi/2\mp u<2\eta,\\
 0,&\text{otherwise}.
 \end{cases}
\]
They extend continuously by zero to $\bar{\angI}$: $k_\eta$ vanishes near the endpoints and agrees with $\rho$ at each inner edge, where $\rho''+\rho=0$.

The endpoint-speed error is controlled by the integral of the absolute angle weight. For fixed $\eta$, \eqref{eq:mixedcontact} uses the measure conversion below, and \eqref{eq:angleaverage} gives
\[
 \rho_j^{-1}d\sigma_{g_1}=\rho_j^{-2}d\sigma_{H,j},
 \qquad
 \|\rho^{-2}(k_\eta''+k_\eta)\|_\infty
 \le\frac{C}{\eta\sin^2\eta}\le C\eta^{-3},\qquad
 \|f_\eta^\pm\|_\infty\le C\eta^{-3},
\]
\[
 \begin{aligned}
 \int_{\Sigma_j}\rho_j^{-2}|(k_\eta''+k_\eta)(u_j)|\,d\sigma_{H,j}
 &\le\frac2\pi P_j\int_{\angI}|k_\eta''+k_\eta|\,du\\
 &\quad+C\|\rho^{-2}(k_\eta''+k_\eta)\|_\infty
       \sqrt{P_j\EE_j}\\
 &\le CP_j+C\eta^{-3}\sqrt{P_j\EE_j},\\
 \limsup_j\int_{\Sigma_j}\rho_j^{-2}
       |(k_\eta''+k_\eta)(u_j)|\,d\sigma_{H,j}
 &\le C\sup_jP_j\le C.
 \end{aligned}
\]
The angular average of each signed endpoint weight gives
\[
 \begin{aligned}
 &\left|\int_{\Sigma_j}f_\eta^\pm(u_j)
       \varphi_j(\mathbf a_j)h(\mathbf a_j^\pm)\,d\sigma_{H,j}
       -\frac2\pi
       \left(\int_{\angI}\rho^2f_\eta^\pm\,du\right)
       \int_{\Sigma_j}\varphi_j(\mathbf a_j)h(\mathbf a_j^\pm)\,d\sigma_{H,j}
       \right|\\
 &\quad\le C\|\varphi_j(\mathbf a)h(\mathbf a^\pm)\|_{C^1(K_0)}
       \|f_\eta^\pm\|_\infty\sqrt{P_j\EE_j}\\
 &\quad\le C_h\eta^{-3}\sqrt{P_j\EE_j}\longrightarrow0.
 \end{aligned}
\]
Here $\int_{\angI}\rho^2f_\eta^\pm\,du=1+O(\eta^2)$ by the endpoint integrals above. Combining the endpoint-speed error, the angular-average error and this $O(\eta^2)$ error in \eqref{eq:mixedcontact} gives
\[
 \begin{aligned}
 &\left|\mathcal B_j(w_{j,\eta},v_j)
       +\frac2\pi\int_{\Sigma_j}\varphi_j(\mathbf a_j)
           [h(\mathbf a_j^+)-h(\mathbf a_j^-)]\,d\sigma_{H,j}\right|\\
 &\quad\le C_h\eta\int_{\Sigma_j}\rho_j^{-2}
       |(k_\eta''+k_\eta)(u_j)|\,d\sigma_{H,j}
       +C_h\eta^{-3}\sqrt{P_j\EE_j}+C_h\eta^2P_j+o_{j,\eta}(1)\\
 &\quad\le C_h\eta P_j+C_h\eta^{-3}\sqrt{P_j\EE_j}+o_{j,\eta}(1).
 \end{aligned}
\]
The constants $C_h$ are independent of $\eta$, while $o_{j,\eta}(1)\to0$ as $j\to\infty$ for each fixed $\eta$. Since $\EE_j\to0$ and $\sup_jP_j<\infty$, taking the limsup gives \eqref{eq:endpointapprox}.
\end{proof}

\subsection{Concentration of centres and identification of the limit}

\Needspace{12\baselineskip}
\begin{lemma}[Concentration of reconstructed centres]\label{lem:centreconcentration}
Let $F_j$ and $K_0$ be as in Lemma~\ref{lem:localpairings}, and suppose $\inf_jP_j>0$.
Assume that the form $\mathcal Q_j$ in \eqref{eq:fullform} is nonnegative on all smooth normal speeds for coefficients $c_j\ge0$ with $\sup_jc_j<\infty$.
Define the horizontal-perimeter averages by
\[
 \bar{\mathbf a}_j=(\bar x_{0,j},\bar y_{0,j},\bar t_{0,j})
 =P_j^{-1}\int_{\Sigma_j}\mathbf a_j\,d\sigma_{H,j}.
\]
Then
\begin{equation}\label{eq:centrevariance}
 \int_{\Sigma_j}|\mathbf a_j-\bar{\mathbf a}_j|_{\mathbb R^3}^2\,d\sigma_{H,j}
 \longrightarrow0.
\end{equation}
\end{lemma}

\begin{proof}
Since $K_0$ is a Euclidean ball, $\bar{\mathbf a}_j\in K_0$.
To recover the three centre coordinates from the endpoint difference, choose $h_{xt},h_{yt},h_{t^2}\in C_c^\infty(\HH)$ agreeing respectively with $xt$, $yt$ and $t^2$ near all points $(x_0,y_0,t_0\pm\pi/4)$ for centres in $K_0$. Their endpoint differences are
\[
 \begin{aligned}
 h_{xt}(\mathbf a^+)-h_{xt}(\mathbf a^-)&=\tfrac\pi2x_0,\\
 h_{yt}(\mathbf a^+)-h_{yt}(\mathbf a^-)&=\tfrac\pi2y_0,\\
 h_{t^2}(\mathbf a^+)-h_{t^2}(\mathbf a^-)&=\pi t_0.
 \end{aligned}
\]
Pair each Hamiltonian with the corresponding coordinate minus its horizontal-perimeter average. Thus $(\varphi_j,h)$ is one of the three pairs
\[
 (x_0-\bar x_{0,j},h_{xt}),\qquad
 (y_0-\bar y_{0,j},h_{yt}),\qquad
 (t_0-\bar t_{0,j},h_{t^2}).
\]
In each case $\int_{\Sigma_j}\varphi_j(\mathbf a_j)\,d\sigma_{H,j}=0$, and the $C^1(K_0)$ norms of $\varphi_j$ and $\varphi_j^2$ are uniformly bounded. The endpoint differences and zero averages turn these pairings directly into variances. For $(\varphi_j,h)=(x_0-\bar x_{0,j},h_{xt})$,
\[
 \begin{aligned}
 \frac2\pi\int_{\Sigma_j}\varphi_j(\mathbf a_j)
       [h_{xt}(\mathbf a_j^+)-h_{xt}(\mathbf a_j^-)]\,d\sigma_{H,j}
 &=\int_{\Sigma_j}(x_0-\bar x_{0,j})x_0\,d\sigma_{H,j}\\
 &=\int_{\Sigma_j}(x_0-\bar x_{0,j})^2\,d\sigma_{H,j}
       +\bar x_{0,j}\int_{\Sigma_j}(x_0-\bar x_{0,j})\,d\sigma_{H,j}\\
 &=\int_{\Sigma_j}(x_0-\bar x_{0,j})^2\,d\sigma_{H,j}.
 \end{aligned}
\]
The $y_0$ pairing equals its variance in the same way, while the $t_0$ pairing equals twice its variance because $h_{t^2}(\mathbf a^+)-h_{t^2}(\mathbf a^-)=\pi t_0$.
For each fixed $\eta$, Lemmas~\ref{lem:cutoffpairings} and \ref{lem:endpointpairings}, together with these identities, give
\begin{equation}\label{eq:endpointcentre}
 \limsup_j\int_{\Sigma_j}|\mathbf a_j-\bar{\mathbf a}_j|_{\mathbb R^3}^2\,d\sigma_{H,j}
 \le C(\sqrt\eta+\eta),
\end{equation}
where $C$ is independent of $\eta$. Letting $\eta\downarrow0$ proves \eqref{eq:centrevariance}.
\end{proof}

\begin{proposition}[Rigidity of the limiting region]\label{lem:stablelimit}
Let $\eps_j\downarrow0$ and let $F_j$ be smooth regions in a common compact set $K\Subset\mathcal U$. Suppose $F_j\to F$ in $L^1$, $|F|>0$, $\sup_jP_H(F_j)<\infty$, and $\partial F_j$ is stationary for
\[
 P_{\eps_j}(G)-\int_G(2+r_j)\,d\mathcal L^3,
 \qquad \|r_j\|_{C^1(\mathcal U)}\longrightarrow0.
\]
Assume also that $\EE_j+\NN_{\eps_j}(F_j)\to0$. Write $\Sigma_j=\partial F_j$, $\NN_j=\NN_{\eps_j}(F_j)$ and $\mathcal B_j=\mathcal B_{2+r_j}$. Suppose that, for coefficients $c_j\ge0$ with $\sup_j c_j<\infty$, the form $\mathcal Q_j$ in \eqref{eq:fullform}
is nonnegative on all smooth normal speeds.
Then $F$ agrees almost everywhere with $\mathbf a_0\cdot\Omega$ for some $\mathbf a_0\in\HH$.
\end{proposition}

\begin{proof}
Write $P_j=P_H(F_j)$. Since $|F_j|\to|F|>0$, after discarding finitely many terms equation~\eqref{eq:profile} gives $P_j\ge I(|F|/2)>0$.
Choose a fixed compact Euclidean ball $K_0$ containing all reconstructed centres.
Lemma~\ref{lem:centreconcentration} gives \eqref{eq:centrevariance} for the averages defined there.
The averages lie in $K_0$, so a subsequence satisfies $\bar{\mathbf a}_j\to\mathbf a_0\in K_0$.

We show that the horizontal distributional derivatives of $\mathbf1_F$ vanish on $\HH\setminus\partial(\mathbf a_0\cdot\Omega)$ by testing horizontal vector fields supported there.
Let $W$ be a smooth compactly supported horizontal vector field whose support misses $\partial(\mathbf a_0\cdot\Omega)$. By continuity of left translation and compactness of $\partial\Omega$, there is a neighbourhood $U_0$ of $\mathbf a_0$ such that every sphere $\partial(\mathbf a\cdot\Omega)$ with $\mathbf a\in U_0$ misses $\supp W$. Equation~\eqref{eq:incidence} gives
\[
 \left|\int_{\Sigma_j}\langle W,N_j\rangle\,d\sigma_{g_1}\right|
 =\left|\int_{\Sigma_j}\langle W,\nu_j\rangle\,d\sigma_{H,j}\right|
 \le\|W\|_\infty\,
       \sigma_{H,j}\{\mathbf a_j\notin U_0\}\longrightarrow0.
\]
Indeed, choose $B_r(\mathbf a_0)\subset U_0$. For all sufficiently large $j$, $|\bar{\mathbf a}_j-\mathbf a_0|<r/2$, and
\[
 \begin{aligned}
 \sigma_{H,j}\{\mathbf a_j\notin U_0\}
 &\le\sigma_{H,j}\{|\mathbf a_j-\mathbf a_0|\ge r\}\\
 &\le\sigma_{H,j}\{|\mathbf a_j-\bar{\mathbf a}_j|\ge r/2\}\\
 &\le\frac4{r^2}\int_{\Sigma_j}
       |\mathbf a_j-\bar{\mathbf a}_j|^2\,d\sigma_{H,j}\longrightarrow0.
 \end{aligned}
\]
By the divergence theorem, the preceding flux is $\int_{F_j}\operatorname{div}W\,d\mathcal L^3$.
Moreover,
\[
 \left|\int_F\operatorname{div}W\,d\mathcal L^3
       -\int_{F_j}\operatorname{div}W\,d\mathcal L^3\right|
 \le\|\operatorname{div}W\|_\infty\,|F_j\triangle F|\longrightarrow0.
\]
Hence
$\int_F\operatorname{div}W\,d\mathcal L^3=0$.
Taking $W=\zeta X$ and $W=\zeta Y$ gives $X\mathbf1_F=Y\mathbf1_F=0$ as distributions off this sphere. The identity $[X,Y]=-2T$ then gives $T\mathbf1_F=0$ there. Since $\partial_x=X-yT$ and $\partial_y=Y+xT$, all coordinate derivatives vanish. The ball and its exterior are connected, so $\mathbf1_F$ is constant on each. Finite volume makes the exterior constant zero; positive volume makes the interior constant one. The sphere has Haar measure zero, and hence $F=\mathbf a_0\cdot\Omega$ almost everywhere.
\end{proof}

\begin{proof}[Proof of Theorem~\ref{thm:main}]
Proposition~\ref{prop:stableapprox} satisfies the hypotheses of Proposition~\ref{lem:stablelimit} by \eqref{eq:stableapprox} and \eqref{eq:fullstability}. Hence every normalized minimizer $E$ is a left translate of $\Omega$. Equation~\eqref{eq:unitball} gives $P=\pi^2$ and $V=3\pi^2/8$.

Existence at every volume and dilation invariance identify the minimizers as exactly the left translates and dilates of $\Omega$. By \eqref{eq:unitball}, for every $W>0$,
\begin{equation}\label{eq:sharp-profile}
 I(W)^4=\frac{512\pi^2}{27}W^3.
\end{equation}
For every measurable set $G$ with $0<|G|<\infty$ and $P_H(G)<\infty$, the definition of the profile gives $P_H(G)\ge I(|G|)$, and hence \eqref{eq:sharp}. Equality makes $G$ a minimizer, so the classification above identifies it, up to a Haar-null set, as a left translate and dilate of $\Omega$. Conversely, \eqref{eq:unitball} and the scaling laws give equality for every such translate and dilate.
\end{proof}

\Needspace{.72\textheight}
\linespread{1}\selectfont
\setlength{\parskip}{0pt}

\end{document}